\documentclass[11pt]{article}

\usepackage{amsmath}
\usepackage{amssymb}
\usepackage{amsthm}
\usepackage{graphicx}
\usepackage{color}

\usepackage{cancel}

\DeclareSymbolFont{calletters}{OMS}{cmsy}{m}{n}
\DeclareSymbolFontAlphabet{\mathcal}{calletters}

\newtheorem{Theorem}{Theorem}[section]
\newtheorem{Definition}[Theorem]{Definition}
\newtheorem{Proposition}[Theorem]{Proposition}

\newtheorem{Assumption}[Theorem]{Assumption}

\newtheorem{Lemma}[Theorem]{Lemma}

\newtheorem{Remark}[Theorem]{Remark}

\def \E{\mathbb{E}}
\def \F{\mathbb{F}}

\def \P{\mathbb{P}}

\def \R{\mathbb{R}}

\def \N{\mathbb{N}}

\def \K{\mathbb{K}}

\def\Ac{{\cal A}}

\def\Cc{{\cal C}}
\def\Dc{{\cal D}}
\def\Ec{{\cal E}}
\def\Fc{{\cal F}}

\def\Hc{{\cal H}}

\def\Kc{{\cal K}}
\def\Lc{{\cal L}}
\def\Mc{{\cal M}}
\def\Nc{{\cal N}}

\def\Pc{{\cal P}}
\def\Qc{{\cal Q}}
\def\Rc{{\cal R}}

\def\Tc{{\cal T}}
\def\Uc{{\cal U}}
\def\Vc{{\cal V}}

\def\Pb{\bar{\P}}

\def\Eb{\overline{\E}}

\def \Om{\Omega}
\def \om{\omega}
\def \Omb{\bar{\Omega}}
\def \omb{\bar{\om}}

\def \eps{\varepsilon}

\def \0{\mathbf{0}}

\newcommand{\we}{\wedge}

\newcommand{\ba}{\begin{array}}
\newcommand{\ea}{\end{array}}
\newcommand{\bea}{\begin{eqnarray}}
\newcommand{\eea}{\end{eqnarray}}
\newcommand{\beaa}{\begin{eqnarray*}}
\newcommand{\eeaa}{\end{eqnarray*}}

\def\be{\begin{eqnarray}}
\def\ee{\end{eqnarray}}
\def\be*{\begin{eqnarray*}}
\def\ee*{\end{eqnarray*}}

\def\g{\gamma}
\def\d{\delta}

\def\l{\lambda}

\def\t{\tau}

\def\cC{{\cal C}}
\def\cD{{\cal D}}

\def\cF{{\cal F}}

\def\cH{{\cal H}}

\def\cL{{\cal L}}

\def\cN{{\cal N}}

\def\cT{{\cal T}}

\def\cV{{\cal V}}

\def\no{\noindent}

\def\q{\quad}

\def\cd{\cdot}

\def\qed{ \hfill \vrule width.25cm height.25cm depth0cm\smallskip}

\newcommand{\basa}{\begin{assumption}}
\newcommand{\easa}{\end{assumption}}

\newcommand{\bas}{\begin{assum}}
\newcommand{\eas}{\end{assum}}

 \def\cd{\cdot}

\def\1{{\bf 1}}

\def\:{\!:\!}

\newtheorem{thm}{Theorem}[section]

\newtheorem{eg}[thm]{Example}

\newtheorem{assum}[thm]{Assumption}

\newcommand{\rmi}{{\rm (i)$\>\>$}}

\newcommand{\rmii}{{\rm (ii)$\>\>$}}
\newcommand{\rmiii}{{\rm (iii)$\>\>$}}
\newcommand{\rmiv}{{\rm (iv)$\>\>$}}

\def\x{\times}

\def\1{{\bf 1}}

\def\no{\noindent}

\def\Xb{\overline{X}}

\def\eg{\textit{e.g.}}
\def\ie{\textit{i.e.}}

\newcommand{\subf}[2]{%
  {\small\begin{tabular}[t]{@{}c@{}}
  #1\\#2
  \end{tabular}}%
}

\title{Mean Field Games with Branching
}

\author{
	Julien Claisse \thanks{Universit\'e Paris-Dauphine, PSL University, CNRS, CEREMADE, Paris. claisse@ceremade.dauphine.fr}
	\and Zhenjie Ren\thanks{Universit\'e Paris-Dauphine, PSL University, CNRS, CEREMADE, Paris. ren@ceremade.dauphine.fr}
        \and Xiaolu Tan\thanks{Department of Mathematics, The Chinese University of Hong Kong. xiaolu.tan@cuhk.edu.hk}
	}
\date{\today}

\begin{document}

\maketitle

\abstract{
	Mean field games are concerned with the limit of large-population stochastic differential games where the agents interact through their empirical distribution. In the classical setting, the number of players is large but fixed throughout the game. However, in various applications, such as population dynamics or economic growth, the number of players can vary across time which may lead to different Nash equilibria. For this reason, we introduce a branching mechanism in the population of agents and obtain 
a variation on the mean field game problem.
	As a first step, we study a simple model using a PDE approach to illustrate the main differences with the classical setting. We prove existence of a solution and show that it provides an approximate Nash-equilibrium for large population games. We also present a numerical example for a linear--quadratic model.  
	Then we study the problem in a general setting by a probabilistic approach. It is based upon the relaxed formulation of stochastic control problems which allows us to obtain a general existence result.

\vspace{1mm}


\vspace{1mm}

}

\section{Introduction}

	The theory of Mean Field Game (MFG) consists in studying
	the limit behaviour of the equilibrium to a stochastic differential game involving a large number of indistinguishable agents,
	who have individually a negligible influence on the overall system and whose decisions are influenced by the empirical distribution of the other agents.
	It was first introduced independently by Lasry and Lions~\cite{lasry2007}, and by Huang \textit{et al.} \cite{huang2006}, in terms of a coupled backward Hamilton--Jacobi--Bellman (HJB) equation and forward Fokker--Planck equation.
	Since then, it has been attracting increasing interest from both the mathematical and the engineering communities.
	Let us refer to the lecture notes of Cardaliaguet~\cite{cardaliaguet2010}  as well as P.-L. Lions' courses on the site of Coll\`ege de France ({\it http://www.college-defrance.fr/site/en-pierre-louis-lions/})
	 for a pedagogical introduction and a detailed overview on this subject.
	
	\vspace{0.5em}
	
	More recently, probabilistic approaches have been developed to study MFG,
	starting with the paper by Carmona and Delarue~\cite{carmona2013probabilistic},
	and have generated a stream of interesting and original results.
	Among them, we would like to mention the weak formulation of MFG introduced by Carmona and Lacker~\cite{CarLac2015} as well as the relaxed formulation introduced by Lacker~\cite{lacker2015}, which yield very general existence results.
	The latter formulation has also lead to a deeper understanding  of the connection between MFG and the corresponding finite player game, by establishing the convergence of $\eps_n$--equilibrium to the $n$-player game toward the MFG solution~\cite{lacker2016}.
	We refer to the book of Carmona and Delarue \cite{CD-book} for a thorough presentation of the probabilistic approach to MFG.
	
	\vspace{0.5em}

	In most of the MFG literature, the corresponding $n$-players game has a constant number of players throughout the game.
	A notable exception occurs in the context of MFG of optimal stopping, where the players choose when they leave the game in an optimal way, so that the population decreases over time.
	We refer to Nutz~\cite{nutz18}, Bertucci~\cite{Ber2017} and Bouveret \emph{et al.}~\cite{bouveret2018mean}	for variations on this problem. Another exception appears in~Campi  and Fischer~\cite{CamFish2018},  where the players quit the game when they reach the boundary of some domain. 	
	Except these interesting examples, a general discussion on MFG allowing the number of players to vary across time is still missing in the literature. 
	This constitutes the first and main objective of this paper.
	Besides its theoretical interest, we believe that this feature can be crucial for various applications in areas such as biology and economy.
	For instance, we might take into account the influence of demography in models of economic growth based on MFG~\cite{gueant11}.
		
	\vspace{0.5em}
	
	Mathematical modelling of population dynamic has been an important topic of research over the last century.
	In particular, the theory of branching processes have been developed  in order to study the evolution of population with random influences, leading to numerous applications in biology and medicine for instance.
	Let us refer to \cite{athreya72,kimmel15} for an introduction to branching processes and their applications in biology.
	In this paper, we are specifically concerned with branching diffusion processes, where each particle has a feature, \eg, its spatial position, whose dynamic is given by a diffusion. 
  It was first introduced by Skorokhod~\cite{skorohod1964branching} and later studied, more thoroughly and systematically, in a series of papers by Ikeda \emph{et al.}~\cite{ikeda1969a}.
 In spite of its potential for applications, the optimal control of branching diffusion processes has not attracted much interest so far. 
 We can mention nonetheless the papers of Ustunel~\cite{ustunel81}, Nisio~\cite{nisio85} and Claisse~\cite{claisse2018}. See also the related work of Bensoussan \emph{et al.}~\cite{Ben2014} on differential games.

	\vspace{0.5em}

	In this paper, we use a branching process to model the evolution of the population of players in the context of MFG.
	To give a brief illustration, let us consider the corresponding finite player game  starting with $n$ initial players. They give rise to a branching diffusion process where each particle corresponds to a player.
	Denote by $K^n_t$ the collection of all players remaining in the game at time $t.$
	The position of player $k\in K^n_t$ follows the controlled dynamic
	\bea \label{eq:diffusion_intro}
		d X^k_t 
		= 
		b\big(t, X^k_t, \mu^n_t, \alpha^k_t\big) \,dt 
		+
		\sigma\big(t, X^k_t, \mu^n_t, \alpha^k_t\big) \,dB^k_t
	\eea
     where $\alpha^k$ corresponds to the strategy of player $k$, $(B^k)_k$ are independent Brownian motions and $\mu^n_t$ is the renormalized empirical occupation measure of the players remaining in the game at time $t$ given as 
     \begin{equation*}
     \mu^n_t := \frac{1}{n} \sum_{k \in K^n_t} \delta_{X^k_t}.
     \end{equation*}
     We consider further that player $k$ leaves the game at a random time $T_k,$ exponentially distributed with intensity $\gamma(t,X_t^k,\mu_t^n,\alpha^k_t),$ and is replaced by $\ell\in\N$ substitute players with probability $p_\ell(T_k,X^k_{T_k},\mu_{T_k}^n).$ 
   In addition, each player $k\in K^n_t$ aims at minimizing his own cost function given as 
	\be*
		 \inf_{\alpha^k} ~ \E  \Big[ \int_t^{T_k\wedge T} f\big(s, X^k_s,  \mu^n_s, \alpha^k_s\big) \, ds  +  g\big( X^k_T, \mu^n_T\big) \mathbf{1}_{T_k\geq T} \Big].
	\ee*
	
	\vspace{0.5em}
	
	As in classical MFG, we heuristically send the number of initial players $n \to \infty,$ so  that 
	the problem can be described by a branching diffusion process starting from one representative player.
	In particular, the limit of the empirical measure is given by
	\be*
		 \mu^n_t = \frac{1}{n} \sum_{k \in K^n_t} \delta_{X^k_t}
		 \xrightarrow[n\to\infty]{}
		 \E \bigg[ \sum_{k \in K^1_t} \delta_{X^k_t} \bigg] =: m_t.
	\ee*
	We highlight that, because of the branching mechanism, the limit measure $m_t$ is not necessarily a probability measure, but a finite positive measure. Finally, the MFG problem in this setting can be defined as follows:
\begin{itemize}
 \item[1.] Fix an environment measure $(\mu_s)_{s\in [0,T]}$ where $\mu_s$ is a finite positive measure on $\R^d$.

 \item[2.] Find a branching diffusion process $(\hat{X}^k_s)_{k\in \hat{K}^1_s}$ where every agent $k$ solves the following stochastic control problem:
	\begin{equation*}
	 \inf_{\alpha^k}~
	  \E \Big[ \int_t^{T_k\wedge T} f(s, X^{k}_s,  \mu_s, \alpha^k_s) \, ds 
		+ g \big( X^{k}_T, \mu_T \big) \1_{T_k \geq T}
	 \Big].
	\end{equation*}
		
 \item[3.] The problem is then to find an equilibrium,
	\ie, an environment measure $(\mu_s)_{s\in[0,T]}$ and an induced optimal branching diffusion process 
	such that 
    \begin{equation*}
     \mu_s = \E \bigg[ \sum_{k \in \hat{K}^1_s} \delta_{\hat{X}^k_s} \bigg] =: \hat{m}_s.
    \end{equation*}	
\end{itemize}

	\vspace{0.5em}

	As a first step, we consider a simple model and we follow the PDE approach described in Cardaliaguet~\cite{cardaliaguet2010} in order to provide a first insight into MFG with branching. 
	We derive the MFG equations as a system of coupled forward--backward PDEs, where the solution to the Fokker--Planck equation takes values in the space of finite measures instead of probability measures. 
	By adapting the arguments in~\cite{cardaliaguet2010}, we are able to ensure existence of a solution and show that it provides an approximate Nash--equilibrium for games involving a large number of players. This result justifies to a certain extent the MFG formulation.
	Then, in the spirit of Bardi~\cite{Bar2012} and Carmona \emph{et al.}~\cite{lachapelle2013} for classical MFG, we focus on a linear--quadratic example where the solution can be computed explicitely by solving a system of ordinary differential equations. This allows us to illustrate numerically the behavior of the MFG equilibrium and the influence of the branching mechanism.

	\vspace{0.5em}
	
	Next we adopt a probabilistic approach and consider a weak formulation of the MFG with branching in a general setting. 
	 We follow the inspiration of  Lacker~\cite{lacker2015} for classical MFG and use the relaxed formulation of stochastic control problem introduced by El Karoui~\textit{et al.}~\cite{elkaroui1987}. 
	More precisely, we introduce a controlled martingale problem on an appropriate canonical space which corresponds to a weak notion of controlled branching diffusion processes. 
	To the best of our knowledge, this relaxed formulation appears for the first time in the literature.
  Then we show existence of a solution under rather general conditions by applying Kakutani's fixed point theorem. 
	In order to provide a transparent presentation of the main techniques for handling the branching mechanism, we restrict to the case of bounded, Markovian coefficient and finite horizon problem, although our results can easily be extended to the case of non-Markovian coefficients with appropriate growth conditions and rather general objective functions.
	
	\vspace{0.5em}
	
	Although sharing some similarities with the arguments in~\cite{lacker2015}, we would like to emphasize that our approach is not a simple extension by replacing the canonical space and the generator of diffusions by those of branching diffusions. 
	Indeed, in our setting, a branching diffusion process is optimal when every particle minimizes its own individual cost function. 
	As their intensity of default depend on the time, position and control variables,
	these optimal particles does not aggregate to form a branching diffusion process which minimizes a global cost function on the whole population. 
	In other words, our problem is not equivalent to a MFG where each agent controls a branching diffusion process $(X^k_t)_{k\in K_t}$  in order to minimize a cost function of the form 
\begin{equation*}
 \E\left[\bar{g}\left(\left(X^k_{T\wedge \cdot}\right)_{k\in K_{T\wedge\cdot}},\mu^n_{T\wedge\cdot}\right)\right].
\end{equation*}	
	For this reason, the  completion of Step~2 above, \ie, the construction of an optimal branching diffusion, becomes rather subtle and technical in the proof of the main existence result.
	This is also the reason why it is not straightforward to extend the arguments in \cite{CarLac2015} to prove uniqueness, or the arguments in~\cite{lacker2016} to derive a general limit theory in our setting.
		These interesting and delicate questions are left for future research.

	\vspace{0.5em}

	The rest of the paper is organized as follows.
	In Section~\ref{sec:formulation}, we provide a detailed mathematical construction of controlled branching diffusion processes and introduce a precise (strong) formulation of the MFG with branching. 
	In Section~\ref{sec:pde}, we apply a PDE approach in a simple framework in order to provide a first insight into this problem. 
	Finally, in Section \ref{sec:probabilistic}, we consider a general setting and we follow a probabilistic approach to formulate and to study a relaxed version of the MFG with branching. 
	These last two sections are independent and we provide a detailed description of their content at the beginning of each of them.
	
%

\paragraph{Notations.}
$\mathrm{(i)}$ Let $X$ be a Polish space, we denote by $\Mc(X)$ (resp. $\Pc(X)$) the Polish space of all finite non-negative measures (resp. probability measures) on $X$, equipped with the weak topology. Denote also by $\Mc_1(X)$ the subspace of measures with finite first moment, where the weak topology can be metrized by the Wasserstein type metric $W_1$ introduced in Appendix~\ref{sec:WasDis_App}.

\noindent $\mathrm{(ii)}$ To describe the genealogy in a branching process, we follow the Ulam--Harris--Neveu notation and give to each particle a label in the following set:
  \begin{equation*}
    \K := \bigcup_{n=1}^{\infty} \N^n.
  \end{equation*}
  Given two labels $k = k_1 \ldots k_n$ and $\tilde{k} = \tilde{k}_1 \ldots \tilde{k}_m$ in $\K$, we define their concatenation as the label $k\tilde{k}:=k_1 \ldots k_n \tilde{k}_1 \ldots \tilde{k}_m$. 
  We write $k\prec \tilde{k}$ (resp. $k\preceq \tilde{k}$) when there exists $k'\in\K$ such that $\tilde{k}=k k'$ (resp. $k\prec \tilde{k}$ or $k=\tilde{k}$).
  
  \noindent $\mathrm{(iii)}$ Let $E$ be the state space of branching diffusion processes defined as
  \begin{equation*}
    E:=\left\{\sum_{k\in K} {\delta_{(k,x^k)}};\ K\subset\K\text{ finite},\, x^k\in\R^d,\, k\nprec \tilde{k}\ \text{for all }k, \tilde{k}\in K \right\}.
  \end{equation*}
  It is a Polish space under the weak topology as a closed subset of $\Mc(\K\x\R^d).$ 

\noindent $\mathrm{(iv)}$ Denote by $\Cc^{2}_b(\K\x\R^d)$ the collection of sequences $\bar{\varphi}=(\varphi^k)_{k\in\K}, \varphi^k\in\Cc^2(\R^d),$ such that $\varphi^k$ and its partial derivatives are bounded uniformly w.r.t. $k\in\K.$

 \noindent $\mathrm{(v)}$ Let $A$ be the control space and denote by $\bar{A}:=A^\K$ the collection of $\bar{a}=(a^k)_{k\in\K}, a^k\in A$. We assume that $A$ is  a nonempty compact metric space throughout the paper.

\section{Formulation of the Problem}
\label{sec:formulation}

\subsection{Controlled Branching Diffusions}
\label{sec:branching-diffusion}

 	A branching diffusion process describes the evolution of a population of independent and identical particles moving according to a diffusion process. 
	In our setting, the dynamic of the particles depends further on a control and an environment measure.  
	More precisely, we consider a population starting with one particle at an initial position with distribution $m_0\in\Pc(\R^d).$ Then the particle moves according to a controlled diffusion with drift $b:[0,T]\x\R^d\x\Mc(\R^d)\x A\to\R^d$ and diffusion coefficient $\sigma:[0,T]\x\R^d\x\Mc(\R^d)\x A\to \R^{d\x d}$.
	Furthermore, the particle dies at rate $\gamma:[0,T]\x\R^d \x\Mc(\R^d)\x A \to\R_+$ and gives birth to $\ell\in\N$ offspring with probability $p_\ell:[0,T]\x \R^d\x\Mc(\R^d)\to [0,1]$ at the position of its death. 
	After their birth, the child particles perform the same dynamic as the parent particle driven by independent Brownian motions and Poisson point processes.

In order to construct the process above, 
we consider a filtered probability space $(\Omega,\cF,(\cF_s)_{s\in [0,T]},\P)$ \
equipped with the following mutually independent sequences of random variables:
\begin{itemize}
\item $(B^{k})_{k\in\K}$ independent $d$-dimensional Brownian motions;
\item $(Q^{k}(ds,dz))_{k\in\K}$ independent Poisson random measures on $[0,T]\x\R_+$ with intensity measure $ds\,dz$.
\end{itemize}
Let $\Ac$ be the collection of $A$-valued predictable processes and denote by $\bar{\Ac}:=\Ac^\K$ the collection of $\bar{\alpha}=(\alpha^k)_{k\in\K}, \alpha^k\in\Ac.$ 
An element $\bar{\alpha}\in\bar{\Ac}$ is a control for the branching diffusion process in such a way that each $\alpha^k\in\Ac$ corresponds to the strategy of particle $k.$

The controlled branching diffusion process is then constructed as follows: Given a control $\bar{\alpha}=(\alpha^k)_{k\in\K}\in\bar{\Ac}$, an environment measure $(\mu_s)_{s\in [0,T]}$, $\mu_s\in\Mc(\R^d)$, and a collection $(X_0^1,X_0^2,\ldots,X_0^n)$ of identical and independent $\Fc_0$--random variables with distribution $m_0\in\Pc(\R^d)$,

\begin{enumerate}
\item Start from $n$ particles at position $X_0^1,X_0^2,\ldots,X_0^n$ and index them by $1,2,\ldots,n$ respectively. These are the particles of generation $1$. We denote by $S_{i}:=0$ the time of birth of particle $i$ for each $i=1,\ldots,n$.

\item For generation $m\geq 1,$ 
provided that the particle $k\in\N^m$ was born at time $S_k<T$,
its dynamic is given as follows:
\begin{itemize}
\item The position $X^k$ of the particle during its lifetime is given by
\begin{equation}\label{eq:eds-k}
X^k_s = X^{k}_{S_k} + \int^s_{S_k} b\big(r,X^k_r,\mu_r,\alpha^k_r\big) \,dr + \int^s_{S_k} \sigma\big(r,X^k_r,\mu_r,\alpha^k_r\big) \,dB^k_r, \quad \P-\text{a.s.}
\end{equation}

\item The time of death/default of the particle is given as 
\begin{equation} \label{eq:def_Tk}
	T_k:=\inf \big\{s > S_k ;\, Q^k\big(\{s\}\x\big[0,\gamma\big(s,X_s^k, \mu_s, \alpha^k_s\big)\big]\big)=1 \big\}\wedge T.
\end{equation}

\item If $T_k<T$, the particle dies and gives rise to $\ell\in\N$ particles provided that 
\begin{equation}\label{eq:def_Uk}
	\sum_{i=0}^{\ell-1} {p_i\big(T_k,X^k_{T_k},\mu_{T_k}\big)} \leq \frac{U_k}{\gamma\big(T_k,X^k_{T_k}, \mu_{T_k}, \alpha^k_{T_k} \big)} < \sum_{i=0}^{\ell} {p_i\big(T_k,X^k_{T_k},\mu_{T_k}\big)},
\end{equation}
where $U_k$ is a positive random variable such that $(T_k,U_k)$ belongs to the support of $Q^k,$ see Remark~\ref{rem:PoissonMeasue}.
These particles belong to the $(m+1)$th generation. 
We index them by label $ki$ and we set $S_{ki}:=T_k$ and $X^{ki}_{S_{ki}}:=X^{k}_{T_k}$ for their time and position of birth for $i = 1,\ldots, \ell.$
\end{itemize}
\end{enumerate}

\begin{Remark} \label{rem:PoissonMeasue}
\rm
    Recall that a Poisson measure $Q(ds,dz)$ on $[0,T]\x[0,C]$ with intensity $ds\,dz$ admits the following representation:
    \begin{equation*}
     Q(ds,dz) = \sum_{i\in\N} \delta_{(\eta_i, \zeta_i)}(ds,dz) \mathbf{1}_{\eta_i\leq T},
    \end{equation*}
    where $(\eta_i-\eta_{i-1}, \zeta_i)_{i\in\N}$ are i.i.d. pairs of random variables with distribution $\Ec(C) \otimes \Uc[0,C]$\footnote{Here $\Ec(C)$ denotes the exponential distribution with parameter $C$, and $\Uc[0,C]$ denotes the uniform distribution on $[0,C]$.}.
    In particular, in view of~\eqref{eq:def_Tk}--\eqref{eq:def_Uk} and Assumption~\ref{hyp:pop} below, the particle $k$ gives rise to $\ell\in\N$ offspring with conditional probability $p_{\ell}(T_k, X^k_{T_k},\mu_{T_k})$ given $\cF_{T_k-}$.
\end{Remark}

  We represent the controlled branching diffusion above as a measure-valued process:
  \begin{equation*}
    Z_s = \sum_{k\in K^n_s} {\delta_{(k,X^k_s)}},
  \end{equation*}
  where $K^n_s$ contains the labels of particles alive at time $s$. 
  Under the following assumptions, the proposition below ensures  that the population process is well-defined.
  
  \begin{Assumption}\label{hyp:pop}   
	\rmi $b$ and $\sigma$ are bounded and there exists $C>0$ such that for all $s\in [0,T]$, $x,y\in\R^d$, $\mu\in\Mc(\R^d)$, $a\in A$,
      \begin{equation*}
	\big|b\left(s,x,\mu,a\right)-b\left(s,y,\mu,a\right)\big| + \big|\sigma\left(s,x,\mu,a\right)-\sigma\left(s,y,\mu,a\right)\big| \leq C \left|x-y\right|;
      \end{equation*}
    \rmii $\gamma$ and  $\sum_{\ell\in\N} {\ell p_{\ell}}$ are bounded.
  \end{Assumption}

  \begin{Proposition}\label{prop:def}
    Under Assumption~\ref{hyp:pop}, there exists a unique (up to indistinguishability) c\`adl\`ag and adapted process $(Z_s)_{s\in [0,T]}$ valued in $E$. In addition, it holds 
    \begin{equation}\label{eq:moment}
      \E\Big[\sup_{s\in[0,T]}{\left\{N^n_s\right\}}\Big]\leq n \exp{\bigg( \Big\|\gamma\sum_{\ell\in\N}{\ell p_{\ell+1}}\Big\|\, T\bigg)},
    \end{equation}
    where $N^n_s:=\# K^n_s$ is the number of particles alive at time $s$.
  \end{Proposition}
  
	The proof of the proposition above can be found in Claisse~\cite[Proposition 2.1]{claisse2018}. It relies essentially on two arguments which follow from Assumption~\ref{hyp:pop}. First, Point~(i) ensures that there exists a unique solution to SDE~\eqref{eq:eds-k}. Second, Assertion (ii) rules out explosion, \ie, there is almost surely finitely many particles in finite time.

	To conclude this section, we derive a semimartingale decomposition for an important class of functionals.
For $\Phi\in\Cc^2(\R),$ $\bar{\varphi}=(\varphi^k)_{k\in\K}, \varphi^k\in\Cc^2(\R^d)$, we define $\Phi_{\bar{\varphi}}:E\to\R$ as for all $e=\sum_{k\in K}{\delta_{(k,x^k)}}$,
\begin{equation*}
 \Phi_{\bar{\varphi}}\left(e\right) := \Phi\big(\langle e, \bar{\varphi}\rangle\big) =  \Phi\Big(\sum_{k\in K}\varphi^k(x^k)\Big).
\end{equation*}
Given $\bar{a}=(a^k)_{k\in\K}, a_k\in A$, we denote by $\Hc^{\bar{a}}$ the operator acting on the class of functions $\Phi_{\bar{\varphi}}$ given by
\begin{multline}
\label{eq:bd-generator}
 \Hc^{\mu,\bar{a}}_s {\Phi_{\bar{\varphi}}} \left(e\right) := \frac{1}{2} \Phi_{\bar{\varphi}}''\left(e\right) \sum_{k\in K} {\big|D {\varphi^k} (x^k) \sigma(s,x^k, \mu_s, a^k) \big|^2} + \Phi_{\bar{\varphi}}'\left(e\right) \sum_{k\in K} {\Lc^{\mu,a^k}_s {\varphi^k}(x^k)} \\
 + \sum_{k\in K} { \gamma(s,x^k, \mu_s, a^k) \left(\sum_{\ell\in\N} { \Phi_{\bar{\varphi}}\Big(e - \delta_{(k,x^k)}  +\sum_{i=1}^{\ell} \delta_{(ki,x^k)} \Big) p_\ell(s,x^k,\mu_s)  - \Phi_{\bar{\varphi}}\left(e\right) } \right) },
\end{multline}
	where $\cL^{\mu,a}_s$ is the infinitesimal generator associated to a diffusion with coefficients 
	$(b,\sigma)$ given as
\begin{equation} \label{eq:Lc}
   \Lc^{\mu,a}_s {\varphi}(x) := \frac{1}{2}\mathrm{tr}\left(\sigma\sigma^*\left(s,x,\mu_s,a\right) D^2{\varphi}\left(x\right)\right)  + b\left(s,x,\mu_s,a\right) \cdot D{\varphi}\left(x\right).
\end{equation}

\begin{Proposition}\label{prop:dsm}
	Under Assumption \ref{hyp:pop},
    	let $\bar{\varphi}\in\Cc^{2}_b(\K\x\R^d)$ and let $\Phi\in\Cc^2_b(\R)$ or $\Phi=\mathrm{Id}$, then the process
  \begin{equation*}
   \Phi_{\bar{\varphi}}\left(Z_s\right)
   - \int_0^s {\Hc^{\mu,\bar{\alpha}_{r}}_r {\Phi_{\bar{\varphi}}} \left(Z_{r}\right) \,dr},\quad s\in [0,T],
  \end{equation*}
	 is a c\`adl\`ag martingale.
\end{Proposition}

The proposition above follows from It\^o's formula, see~\cite[Proposition 3.2]{claisse2018}. It is the starting point for the relaxed formulation which is introduced in Section~\ref{sec:probabilistic}.

\subsection{Mean Field Games with Branching}
\label{subsec:MFG_descrip}

To describe the MFG of interest, let us start from the situation with a finite number of agents, say $n\in\N$ initial agents at time $0$.
They generate a branching diffusion process as described above where each particle $k$ corresponds to an agent entering the game at time $S_k$ and leaving at time $T_k$. 
We assume further that the dynamics of the agents are coupled through their empirical distribution\footnote{Another formulation consists of writing $\mu^n_s =\frac{1}{N^n_s} \sum_{k \in K^n_s} \delta_{X^{k}_s}.$ 
However our formulation allows for a broader range of interactions without making the analysis more complex.
}
$\mu^n$ given by
\begin{equation*}
 \mu^n_s ~:=~ \frac{1}{n} \sum_{k \in K^n_s} \delta_{X^{k}_s}.
\end{equation*}
In addition, each agent $k$ applies a strategy $\alpha^{k}\in\Ac$ in order to minimize the following cost criterion:
\begin{equation} \label{eq:reward_k}
	 \E \Big[ \int_{S_k}^{T_k} f(s, X^{k}_s, \mu^n_s, \alpha^{k}_s)  \,ds + g( X^{k}_T, \mu^n_T)\mathbf{1}_{T_k=T} \Big].
\end{equation}
where $f:[0,T]\x\R^d\x\Mc(\R^d)\x A\to\R$ and $g:\R^d\x\Mc(\R^d)\to\R$ are bounded. 
Namely, the agent pays a running cost $f$ while it is in the game and a terminal cost $g$ if it stays until time $T$.

\begin{Remark}
\rm
	It would be natural to also take into account a terminal cost when an agent leaves the game before time $T$, \ie, to add to~\eqref{eq:reward_k} a term of the form
\begin{align*}
 \E\Big[\tilde{g}\big(T_k,X^k_{T_k},\mu^n_{T_k}\big)\mathbf{1}_{T_k<T}\Big]
 & = \E\Big[\int_{(S_k,T_k]\x\R_+} {\tilde{g}\big(s,X^k_s,\mu^n_s\big)\1_{z\leq \gamma(s,X^k_s, \mu^n_s, \alpha^k_s)} \,Q^k(ds,dz)}\Big] \\
 & = \E\Big[\int_{S_k}^{T_k} {\tilde{g}\big(s,X^k_s,\mu^n_s\big) \gamma\big(s,X^k_s, \mu^n_s, \alpha^k_s \big)\,ds}\Big].
\end{align*}
In particular, we see that this feature is already encompassed in our formulation by a suitable modification of the running cost $f$.
\end{Remark}

\begin{Remark}
\rm
	 Actually the agent $k$ entering the game at time $S_k$ can see the current state of the game and is rather aiming to minimize
 \begin{equation*}
	\E \Big[ \int_{S_k}^{T_k} f(s, X^{k}_s, \mu^n_s, \alpha^{k}_s)  \,ds + g( X^{k}_T, \mu^n_T)\mathbf{1}_{T_k= T} ~\Big|~ \Fc_{S_k} \Big].  
 \end{equation*}
	By a classical dynamic programming argument, it turns out that it is equivalent to the optimization problem~\eqref{eq:reward_k}, in the sense that an optimal strategy for one remains optimal  for the other.
\end{Remark}

Similar to the classical MFG approach, as the initial number of agents $n \to \infty$, the influence of a single agent on the empirical distribution becomes negligible and so each agent can consider the empirical distribution as fixed.
Additionally, by construction, the entire game is symmetric and so we expect that all the players use the same optimal strategy.
We further expect by a law of large number type argument that the sequence of empirical measures $(\mu^n)_{n\in\N}$ converges to a limit $\mu$ given as
\begin{equation*}
 \langle \mu^n_s, \varphi\rangle = \frac{1}{n} \sum_{i=1}^n {\sum_{k\in K^{n,i}_s} {\varphi(X_s^k)}} \longrightarrow  \E\bigg[\sum_{k \in K^1_s} \varphi(X^{k}_s)\bigg] =: \langle \mu_s, \varphi\rangle,
\end{equation*}
where $K^{n,i}_s:=\{k\in K^n_s;\, k\succeq i\}$ denotes the collection of descendants of agent $i$.

	Hence, the limit system can formally be described by a branching diffusion process starting from a single agent and we obtain the following new mean field game problem:

\begin{itemize}
 \item[1.] Fix $(\mu_s)_{s\in [0,T]}$, $\mu_s\in\Mc(\R^d)$, an environment measure.

 \item[2.] \label{step:hjb}
 	Find a branching diffusion process $\hat{Z}_s=\sum_{k\in \hat{K}^1_s}{\delta_{(k,\hat{X}^k_s)}}$ where every agent $k$ solves the following stochastic control problem:
	\begin{equation*}
	 \inf_{\alpha^k\in\Ac}~
	  \E \Big[ \int_{S_k}^{T_k} f(s, X^{k}_s,  \mu_s, \alpha^k_s) \, ds 
		+ g \big( X^{k}_T, \mu_T \big) \1_{T_k =T}
	 \Big].
	\end{equation*}
		
 \item[3.] \label{step:fk}
 	The problem is then to find an equilibrium,
	\ie, an environment measure $(\mu_s)_{s\in[0,T]}$ and an induced optimal branching diffusion process $\hat{Z}$
	such that for all $\varphi:\R^d\to\R$ bounded,
    \begin{equation*}
    \langle\mu_s,\varphi\rangle = \E\bigg[\sum_{k \in \hat{K}^1_s} \varphi\big(\hat{X}^k_s\big)\bigg].
    \end{equation*}	
\end{itemize}

In the next section, we study this problem for a simple model using a PDE approach to provide a first insight into the MFG with branching. Then we follow a probabilistic approach in a general setting by considering a weak formulation of this problem, based upon the relaxed formulation of stochastic control problem.

\section{PDE Approach}
\label{sec:pde}

  In this section, we consider a simple model and we follow the PDE approach described in Cardaliaguet~\cite{cardaliaguet2010} in order to provide a first insight into MFG with branching. 
   First we introduce the corresponding MFG equations in Section~\ref{subsec:MFG_PDE} and we establish existence of a solution in Theorem~\ref{thm:mfgedp}. 
   The proof is completed in Section~\ref{subsec:DeductionMFGEquation}, where we provide a detailed analysis of the MFG equations.
   Then we show in Section~\ref{subsec:nash} that a solution to the MFG with branching provides an approximate Nash-equilibrium for large population games, thus justifying the MFG formulation. 
	Finally, we study a numerical example in Section~\ref{subsec:numeric} to illustrate the behaviour of the equilibrium of the MFG with branching.

\subsection{PDE Formulation}
\label{subsec:MFG_PDE}

 In the simple framework of this section, we consider the following parameters for the model: $f:\R^d\x\Mc_1(\R^d)\to\R$, $g:\R^d\x\Mc_1(\R^d)\to\R$, $\gamma:\R^d\to \R_+$,  and $(p_\ell)_{\ell\in\N}$, $p_\ell:\R^d\to [0,1]$, such that $\sum_{\ell\in\N} {p_\ell}=1.$
It is implicitly assumed that $b(t,x,\mu,a)=a,$ $\sigma(t,x,\mu,a)=\sqrt{2}$ and $f(t,x,\mu,a)=\frac{1}{2}|a|^2 + f(x,\mu).$
	Then we introduce the following PDE formulation of the MFG with branching:
\begin{eqnarray}
	\partial_t u + \Delta u - \frac{1}{2}\left|Du\right|^2 - \gamma u + f\left(\cdot,m\right) = 0, & &  \text{in }[0,T)\x\R^d, \label{eq:hjb}\\
	\partial_t m - \Delta m - \mathrm{div}\left(m Du\right) - \gamma\sum_{\ell\in\N} {(\ell-1) p_{\ell}}\, m = 0, & & \text{in }(0,T]\x\R^d, \label{eq:fk}\\
	u(T,\cdot) = g\left(\cdot,m(T)\right), ~ m(0) = m_0, & & \text{in }\R^d. \label{eq:bc}
\end{eqnarray}
The derivation of these equations from the MFG with branching introduced in  Section~\ref{subsec:MFG_descrip} is performed in Section~\ref{subsec:DeductionMFGEquation}. 

\begin{Definition}
 A solution to the MFG with branching \eqref{eq:hjb}--\eqref{eq:bc} is a couple $(u,m)\in\Cc^{1,2}([0,T]\x\R^d)\times \cC([0,T],\Mc_1(\R^d))$ such that $u$ is a classical solution to~\eqref{eq:hjb} and $m$ is a weak solution to~\eqref{eq:fk}
 satisfying the boundary condition~\eqref{eq:bc},
in the sense that for all $\varphi\in\cC_c^{\infty}([0,T)\x\R^d)$,
\begin{equation*}	
	\int_0^T \!\!\int_{\R^d} \!\!
	{\Big(\partial_t\varphi + \Delta \varphi - Du \cdot D\varphi + \gamma\sum_{\ell\in\N}{(\ell-1) p_\ell}\,\varphi\Big)(t,x)\,m(t,dx)\,dt}
	+ \int_{\R^d} {\varphi(0,x) \, m_0(dx)} = 0.
\end{equation*}
\end{Definition}
 
Under the following assumptions, we can show that there exists a solution to this problem by extending the arguments in~\cite{cardaliaguet2010}.
The proof is postponed to Section~\ref{sec:existence}. 

\begin{Assumption}\label{ass:cardaliaguet}
\rmi $f$ and $g$ are bounded and there exists  $L>0$ such that for all $x,y\in\R^d$, $\mu,\nu\in\Mc_1(\R^d)$, 
\begin{equation*}
 \big|f(x,\mu) - f(y,\nu)\big| + \big|g(x,\mu) - g(y,\nu)\big|  \leq L \big(|x-y| + W_1(\mu,\nu)\big).
\end{equation*}
\rmii $\gamma$ and $\gamma\sum_{\ell\in\N}{\ell p_\ell}$ are Lipschitz and bounded. \\
\rmiii $g(\cdot,\mu)\in\Cc^3_b(\R^d)$ for all $\mu\in\Mc_1(\R^d)$.\\
\rmiv $\int_{\R^d} {|x|^2 \,m_0(dx)}<\infty$.
\end{Assumption}

\begin{Theorem}
\label{thm:mfgedp}
 Under Assumption~\ref{ass:cardaliaguet}, there exists a solution to the MFG with branching \eqref{eq:hjb}--\eqref{eq:bc}.
\end{Theorem}

\begin{Remark}
\label{rem:uniqueness_edp}
\rm
	$\mathrm{(i)}$ When the coefficients $\gamma$ and $(p_\ell)_{\ell\in\N}$ are constant, the problem reduces to the classical MFG studied in~\cite{cardaliaguet2010} by a simple change of variable. Indeed, the normalized measure
	\begin{equation*}
		\bar m(t) := \exp{\Big(-\gamma\sum_{\ell\in\N}{(\ell-1) p_{\ell}} \,t\Big)}m(t),
	\end{equation*}
	satisfies the classical Fokker--Planck equation
	\begin{equation*}
	 \partial_t \bar m - \Delta \bar m - \mathrm{div}\left(\bar m Du\right) = 0, \quad \text{in }(0,T]\x\R^d.
	\end{equation*}
	
	\noindent $\mathrm{(ii)}$ 
	In view of the above, when $\gamma$ and $(p_\ell)_{\ell\in\N}$ are constant, 
	uniqueness of the solution to~\eqref{eq:hjb}--\eqref{eq:bc} follows under the usual monotonicity conditions:
	for all $\mu, \nu \in \Mc_1(\R^d)$ such that $\mu\neq \nu$ and $\mu(\R^d) = \nu(\R^d),$
	\begin{equation*}
	  \int_{\R^d} {\big(f(x, \mu) - f(x, \nu) \big) \,(\mu-\nu)(dx)} > 0, \quad
	  \int_{\R^d} {\big(g(x,\mu) - g(x,\nu) \big)} \,(\mu-\nu)(dx) \ge 0.
	\end{equation*}
	Otherwise, uniqueness does not follow from a direct extension of the arguments in~\cite{cardaliaguet2010} as we have to deal with an additional term of the form $\int_{\R^d} {(u_1-u_2)\, d(m_1-m_2)}$ whose sign is undetermined.
\end{Remark}

\subsection{Analysis of the MFG Equations}
\label{subsec:DeductionMFGEquation}

\subsubsection{Hamilton--Jacobi--Bellman Equation}
\label{sec:hjb}
   Let us study first the HJB equation~\eqref{eq:hjb} and show that it derives from the optimal control problem solved by each agent in the MFG with branching described in Section~\ref{subsec:MFG_descrip}.
    
	Let $(\Om,\cF,(\cF_s)_{s\in [0,T]},\P)$ be a filtered probability space equipped with a Brownian motion $(B_s)_{s\in [0,T]}$ and a Poisson random measure $Q(ds,dz)$ on $[0,T]\x\R_+$ with intensity $ds\,dz$. 
	Denote by $\Ac$ the collection of $\R^d$--valued predictable processes $\alpha$ such that $ \int_0^T {\left|\alpha_s\right| ds} <\infty.$
	Given an environment measure $(\mu_s)_{s\in[0,T]}, \mu_s\in\Mc_1(\R^d),$ we consider the following optimal control problem:
	\begin{equation}\label{eq:def_v}
		v(t,x,\mu) := \inf_{\alpha\in\Ac} {\big\{J(t,x,\mu,\alpha)\big\}},
	\end{equation}
	with
\begin{equation*}
 J(t,x,\mu,\alpha) 
	  :=  \E\left[\int_t^{\tau} {\left(f(X^{t,x,\alpha}_s,\mu_s) + \frac{1}{2} \left|\alpha_s\right|^2\right)\, ds} + g(X^{t,x,\alpha}_T,\mu_T) \mathbf{1}_{\tau= T}\right],
\end{equation*}
where 
\begin{equation*}
 X^{t,x,\alpha}_s := x + \int_t^s {\alpha_r \,dr} + \sqrt{2} \left(B_s-B_t\right),\qquad s\in [t,T],
\end{equation*}
and 
\begin{equation*}
 \tau := \inf \big\{s > t ;\, Q\big(\{s\}\x [0,\gamma(X^{t,x,\alpha}_s)]\big) =1 \big\}\wedge T.
\end{equation*}
 It corresponds to the optimal control problem solved by each agent when the environment measure $\mu$ is fixed. The next proposition ensures that the value function $v$ solves the HJB equation~\eqref{eq:hjb} with $m$ replaced by $\mu$ and that each agent can use a Markovian optimal control $\hat \alpha(t,x) = - Dv(t,x,\mu).$

\begin{Proposition}
\label{prop:hjb}
	Let Assumption~\ref{ass:cardaliaguet} hold. If we assume further that
	$s\mapsto\mu_s$ is H\"older continuous,
	then the value function $v(\cdot,\mu)$ in \eqref{eq:def_v} belongs to $\Cc^{1,2}_b([0,T]\x\R^d)$ and is the unique (bounded) classical solution to
\begin{equation}\label{eq:HJB}
 \partial_t u + \Delta u - \frac{1}{2}|Du|^2 - \gamma u + f(\cdot,\mu) = 0 ~~\text{in }[0,T)\x\R^d,\quad u(T,\cdot) = g(\cdot,\mu_T) ~~ \text{in }\R^d.
\end{equation}
In addition, $\hat{\alpha}(t,x) := -Dv(t,x,\mu)$ is an optimal Markov control.
\end{Proposition}

\begin{proof}
	Let us write a short proof of this classical result for the sake of completeness.
	
	\noindent\rmi We start by proving that there exists a classical solution to PDE~\eqref{eq:HJB}. To this end, we use the well-known Hopf--Cole transform: setting $w=e^{-\frac{u}{2}}$ we easily check that $u$ is a solution to~\eqref{eq:HJB} if and only if $w$ is a solution to 
\begin{equation}\label{eq:hopf-cole}
 \partial_t w + \Delta w  =  w\left(\frac{1}{2} f(\cdot,\mu)+ \gamma \log(w)\right)  ~~\text{in }[0,T)\x\R^d,\quad w(T,\cdot) = e^{-\frac{g(\cdot,\mu_T)}{2}} ~~ \text{in }\R^d.
\end{equation}
Since $f$ and $g$ are bounded, one can easily find constant lower and upper solutions $(\underline{w},\overline{w})$ such that $0<\underline{w}<\overline{w}<\infty$.
Then existence of a classical solution $\underline{w}\leq w \leq\overline{w}$ to PDE~\eqref{eq:hopf-cole} follows from classical arguments, see, \eg, Pao~\cite[Theorem 7.2.1]{pao1992}. 

\noindent \rmii The conclusion follows by a verification theorem. Indeed, using It\^o's formula and the fact that $u$ satisfies~\eqref{eq:HJB}, we derive that for all $\alpha\in\Ac$,
\begin{equation*}
 J(t,x,\mu,\alpha) \geq u(t,x) + \E\left[\int_t^\tau {\gamma(X^{t,x,\alpha}_s)u(s,X^{t,x,\alpha}_s)\,ds} - u(\tau, X^{t,x,\alpha}_\tau)\1_{\tau<T}\right].
\end{equation*}
Additionally, it holds
\begin{align*}
  \E\left[u(\tau, X^{t,x,\alpha}_\tau)\1_{\tau<T}\right] 
   & =  \E\Big[\int_{(t,\tau]\x\R_+} \!\!\! {u(s,X^{t,x,\alpha}_s) \1_{z\leq \gamma(X^{t,x,\alpha}_s)}\, Q(ds,dz)}\Big] \\
   & =  \E\Big[\int_t^\tau \!\! {\gamma(X^{t,x,\alpha}_s)u(s,X^{t,x,\alpha}_s)\,ds} \Big].
\end{align*}
We deduce that $v(t,x,\mu)\geq u(t,x)$. To conclude, it remains to repeat the same computation with the Markov control $\hat{\alpha}$.
\qed  
\end{proof}

 We conclude this section with an important technical lemma.

\begin{Lemma}
\label{lem:gradient}
 Under the assumptions of Proposition~\ref{prop:hjb}, the gradient $Dv(\cdot,\mu)$ is bounded uniformly w.r.t. $\mu$.
\end{Lemma}

\begin{proof}
 Since $f$ and $g$ are bounded, we can restrict to admissible controls satisfying
 \begin{equation} \label{eq:bounded_ctrl}
 \E\left[\frac{1}{2}\int_0^T e^{-\|\gamma\| s} \left|\alpha_s\right|^2\,ds\right]\leq 2\big(T \|f\| + \|g\|\big).
 \end{equation}
 Indeed, the constant control $\alpha=0$ performs better than any control which does not satisfy this condition. 
	Then it holds that
\begin{equation*}
 \left|v(t,x,\mu) - v(t,y,\mu)\right| ~\leq~ \sup_{\alpha} \big\{\big|J(t,x,\mu,\alpha)-J(t,y,\mu,\alpha)\big|\big\},
\end{equation*}
where the supremum is taken over all controls satisfying \eqref{eq:bounded_ctrl}.
To conclude, it remains to see that the restriction of the cost function to such control is Lipchitz continuous in $x$, uniformly w.r.t. $(t,\mu,\alpha)$. 
\qed
\end{proof}

\subsubsection{Fokker--Planck Equation}
\label{sec:fokker-planck}

 Let us study next the Fokker--Planck equation~\eqref{eq:fk} and show that the distribution of a branching diffusion where every agent uses the optimal control of Proposition~\ref{prop:hjb} satisfies~\eqref{eq:fk} with $Du$ replaced by $Dv(\cdot,\mu).$
 
Let $b:[0,T]\x\R^d\to \R$ be bounded satisfying there exists $C>0$ such that for all $s,t\in[0,T]$, $x,y\in\R^d,$
\begin{equation*}
 \big|b(t,x)-b(s,y)\big|\leq C\Big(|x-y| + \sqrt{|t-s|}\Big).
\end{equation*}
We consider a branching diffusion process, starting at time $0$ from one particle at position $X_0$ with distribution $m_0,$ such that the particles follow a diffusion with drift $b$ and diffusion coefficient $\sqrt{2}$ and die at rate $\gamma$ while giving birth to $\ell\in\N$ particles with probability $p_\ell.$
We denote by $(X^k_t)_{k\in K_t}$ the state of the branching diffusion process at time $t\in[0,T]$ and we define the measure $m(t)$ as follows: for all $\varphi:\R^d\to\R$ bounded,
\begin{equation} \label{eq:def_mt}
	\langle m(t), \varphi\rangle := \E\bigg[\sum_{k\in K_t}\varphi\big(X^k_t\big)\bigg].
\end{equation}
We refer to Section~\ref{sec:branching-diffusion} for more details.
	
\begin{Proposition}
\label{prop:fk}
	Under Assumption~\ref{ass:cardaliaguet}, the map $t\mapsto m(t)$  in~\eqref{eq:def_mt} belongs to $\Cc([0,T],\Mc_1(\R^d))$ and is the unique weak solution to
	\begin{equation}\label{eq:fokker-planck}
		\partial_t m - \Delta m + \mathrm{div}\left(m b\right) - \gamma\sum_{\ell\in\N}{(\ell-1) p_\ell}\, m = 0 ~~ \text{in }(0,T]\x\R^d,\quad m(0) = m_0 ~~ \text{in }\R^d,
	\end{equation}
   in the sense that for all $\varphi\in\cC_c^{\infty}([0,T)\x\R^d),$
	\begin{equation*}		 
		 \int_0^T\!\! \int_{\R^d}\!\!
		 {\Big(\partial_t\varphi + \Delta \varphi + b\cdot D\varphi + \gamma\sum_{\ell\in\N}{(\ell-1) p_\ell}\,\varphi\Big)(t,x)\,m(t,dx)\,dt}
		 + \int_{\R^d} {\varphi(0,x) \, m_0(dx)} = 0. 
	\end{equation*}
\end{Proposition}

\begin{proof}
	Let us show first that $m$ is a weak solution to~\eqref{eq:fokker-planck}. 
		Proposition~\ref{prop:dsm} ensures that, for all $\varphi\in\cC_c^{\infty}([0,T)\x\R^d),$  the process
\begin{equation*}
 \sum_{k\in K_t} {\varphi(t,X^k_t)} - \varphi(0,X_0)
 - \int_0^t {\sum_{k\in K_s} \Big(\partial_t\varphi + \Delta \varphi + b\cdot D\varphi + \gamma\sum_{\ell\in\N}{(\ell-1) p_\ell}\,\varphi\Big)(s,X^k_s) \, ds},
\end{equation*}
is a martingale.
Since $\varphi(T,\cdot)=0$, we deduce that
\begin{equation*}
 \E\big[\varphi(0,X_0)\big]
 + \int_0^T {\E\bigg[\sum_{k\in K_s} \Big(\partial_t\varphi + \Delta \varphi + b\cdot D\varphi + \gamma\sum_{\ell\in\N}{(\ell-1) p_\ell}\,\varphi\Big)(s,X^k_s)\bigg] \,ds} = 0,
\end{equation*}
which is the desired result.
Next the fact that $m$ belongs to $\Cc([0,T],\Mc_1(\R^d))$ follows from Lemma~\ref{lem:technicpde} below. As for uniqueness, it comes from a classical duality argument, see, \eg, Proposition~3.1 in Ambrosio~\textit{et al.}~\cite{ASZ09}.
\qed
\end{proof}

	For further developments, let us collect a couple of properties on the solution to the Fokker--Planck equation. 
	The proof is postponed to Appendix~\ref{app:proof-lemma}.

\begin{Lemma}
\label{lem:technicpde}
	Under the assumptions of Proposition~\ref{prop:fk}, there exists a constant $C>0$ such that for all $s, t\in[0,T]$,
	\begin{equation*}
		 \int_{\R^d} {\left(1 + |x|^2\right) m(t,dx)} \leq C \left(1+\|b\|^2\right)
		 ~~\text{and}~~
		 W_1(m(t),m(s))  \leq C\left(1+\|b\|^2\right)\sqrt{|t-s|}.
	\end{equation*}
\end{Lemma}

\subsubsection{Proof of Theorem~\ref{thm:mfgedp}}
\label{sec:existence}

	The main idea of the proof of Theorem~\ref{thm:mfgedp} is to apply the Schauder fixed point theorem to the function $\psi$ defined as follows: To any $\mu$ in a well-chosen subset of $\Cc([0,T],\Mc_1(\R^d))$, 
	we associate $\psi(\mu)=m$ where $m$ is the solution to PDE~\eqref{eq:fokker-planck} with $b =-Du$ and $u$ is the solution to PDE~\eqref{eq:HJB} corresponding to $\mu$.

	Let $C > 0$ be a fixed constant, we denote by $\Cc$ the collection of all maps $\mu:[0,T]\to\Mc_1(\R^d)$ such that
	\begin{equation*}
	W_1\big(\mu(t),\mu(s)\big)\leq C\sqrt{|t-s|} \quad \text{and} \quad \int_{\R^d} (1+|x|^2) \, \mu(t,dx) \leq C.
	\end{equation*}
	Then $\cC$ is a compact\footnote{The relative compactness follows from Arzel\`a--Ascoli Theorem since $\{\mu(t),\,\mu\in\cC\}$ is relatively compact in $\Mc_1(\R^d)$ for every $t \in [0,T]$, see Lemma \ref{lemm:Wp_compact}.} convex subset of $\cC([0,T],\Mc_1(\R^d))$ equipped with the uniform topology.
	Additionally, in view of Lemmas~\ref{lem:gradient} and~\ref{lem:technicpde}, there exists a suitable choice of constant $C$ such that $\psi(\Cc)\subset\Cc$.

To conclude the proof, it remains to show that $\psi$ is continuous on $\Cc$. We consider a sequence $(\mu_n)_{n\in\N}$ in $\cC$ converging to $\mu$ and denote $m_n:=\psi(\mu_n)$ and $m:=\psi(\mu)$. We aim to show that $m_n$ converges to $m$. Let $u_n$ be the solution to PDE~\eqref{eq:HJB} corresponding to $\mu_n$.
In view of Proposition~\ref{prop:hjb}, $u_n$ admits the probabilistic representation
\begin{equation*}
  u_n(t,x)= v\left(t,x,\mu_n\right).
\end{equation*}
By a straightforward computation, we see that $u_n$ converges uniformly to
\begin{equation*}
  u(t,x)= v\left(t,x,\mu\right),
\end{equation*}
which is the solution to PDE~\eqref{eq:HJB} corresponding to $\mu$. Additionally, $u_n$ satisfies
\begin{equation*}
 \partial_t u_n + \Delta u_n = f_n ~~\text{in }[0,T)\x\R^d,\quad u_n(T,\cdot) = g(\cdot,\mu_n(T)) ~~ \text{in }\R^d,
\end{equation*}
where 
\begin{equation*}
f_n:(t,x)\mapsto \frac{1}{2}|Du_n(t,x)|^2 + \gamma(x) u_n(t,x) - f(x,\mu_n(t)).
\end{equation*}
Since $f_n$ is continuous and uniformly bounded in $n$, Theorem~3.11.1 in Lady\v{z}enskaja \textit{et al.}~\cite{ladyzenskaja1968} ensures that $Du_n$ is locally H\" older continuous uniformly in $n$. Thus 
$(Du_n)_{n\in\N}$ is relatively compact by Arzel\`a-Ascoli theorem and so it converges locally uniformly to $Du$. We deduce that 
any limit point of $(m_n)_{n\in\N}$ is a weak solution to PDE~\eqref{eq:fokker-planck} with $b=-Du$. 
	The conclusion then follows by weak uniqueness.
\qed

\subsection{Approximate Nash Equilibrium}
\label{subsec:nash}

 In the spirit of Section~3.4 in~\cite{cardaliaguet2010}, we can show that the solution to the MFG with branching provides an $\eps$-Nash equilibrium for the corresponding $n$-player game when $n$ is large enough. To a certain extent, this result justifies the MFG formulation as an approximation of games involving a large number of players.

	Consider $n$ initial agents at position $(X^1_0, \cdots, X^n_0)$ 
	which consists in an i.i.d. family of random variables with distribution $m_0$.
	Given a control $\bar{\alpha} = (\alpha^k)_{k \in \K}$, we can construct the controlled branching diffusion process $(X^k_t)_{k\in K^n_t}$ as in Section~\ref{sec:branching-diffusion}, where each agent $k$ follows the dynamic
	\begin{equation*}
		dX^{k}_t = \alpha^k_t \,dt + \sqrt{2} \,dB^k_t,
	\end{equation*}
	while choosing the strategy $\alpha^k$ in order to minimize
	\begin{equation*}
		J^n_k (\bar{\alpha}) :=
		\E\Big[
		  	\int_{S_k}^{T_k} {\Big(f(X^{k}_s, \mu^{n,k}_s) 
		  	+
		  	\frac{1}{2} \big|\alpha^k_s\big|^2\Big)\, ds} + g(X^{k}_T, \mu^{n,k}_T) \mathbf{1}_{T_k= T}
		\Big],
	\end{equation*}
	where $\mu^{n,k}_s := \frac{1}{n} \sum_{k' \in K^n_s\setminus\{k\}} \delta_{X^{k'}_s}$.

	Let us also consider the branching diffusion process $(\hat{X}^k_t)_{k\in\hat{K}^n_t}$, where every agent applies the closed loop strategy given by the MFG with branching:
	\begin{equation*}
		d \hat{X}^{k}_t 
		=
		- Du (t, \hat{X}^{k}_t) \,dt + \sqrt{2} \,dB^k_t,
	\end{equation*}
    where $(u,m)$ satisfies~\eqref{eq:hjb}--\eqref{eq:bc}. The corresponding open loop control $\hat{\alpha}:=(\hat{\alpha}^k)_{k\in\K}$ is given by
    \begin{equation*}
     \hat{\alpha}^k_t := -Du (t, \hat{X}^{k}_t).
    \end{equation*}
    As stated below, it provides an approximate Nash equilibrium for large population games.

	\begin{Theorem}\label{thm:nash}
		Under Assumption \ref{ass:cardaliaguet}, for any $\eps > 0$, there exists $n_0\in\N$ such that for all $n \ge n_0$, 
		the symmetric strategy $\hat{\alpha}$ provides an $\eps$-Nash equilibrium in the sense that
		\begin{equation*}
			J^n_k \left(\hat{\alpha}\right) 
			\leq
			J^n_k \big(\alpha^k , \hat{\alpha}^{-k} \big) + \eps, \qquad \forall\,\alpha^k\in\Ac,\,k\in\K.
		\end{equation*}
	\end{Theorem}

	\begin{proof}
	We break down the computation as follows:
	\begin{equation}\label{eq:nash-proof}
		J_k^n(\hat{\alpha}) - J^n_k \big(\alpha^k , \hat{\alpha}^{-k} \big) 
		\leq 
		\big( J_k (\hat{\alpha}^k) -  J_k (\alpha^k)\big) 
		+
		\big(J_k^n(\hat{\alpha}) - J_k (\hat{\alpha}^k)\big) 
		+ \big(J_k (\alpha^k) - J^n_k(\alpha^k , \hat{\alpha}^{-k})\big),
	\end{equation}
    	where 
	\begin{equation*}
		J_k (\alpha^k) :=
		\E\Big[
		  	\int_{S_k}^{T_k} {\Big(f(X^{k}_s, m_s) 
		  	+
		  	\frac{1}{2} \big|\alpha^k_s\big|^2\Big)\, ds} + g(X^{k}_T, m_T) \mathbf{1}_{T_k= T}
		\Big].
	\end{equation*}
	
	\noindent $\mathrm{(i)}$  The first term on the r.h.s. of~\eqref{eq:nash-proof} is non-positive. Indeed, it holds
	\begin{equation*}
	 J_k (\alpha^k) \geq \E\Big[v\big(S_k, X^k_{S_k},m\big)\Big] = J_k(\hat{\alpha}^k),
	\end{equation*}
	where the inequality follows from a dynamic programming principle and the equality from the optimality of $\hat{\alpha}^k$ established in Proposition~\ref{prop:hjb}.
	
    \noindent $\mathrm{(ii)}$ Next we deal with the second term on the r.h.s. of~\eqref{eq:nash-proof}.
    It holds that
    \begin{equation*}
      \big|J_k^n(\hat{\alpha}) - J_k (\hat{\alpha}^k)\big| \leq C \E\Big[ \int_0^T \big( W_1 (\hat{\mu}^{n,k}_s, m_s) \wedge 1 \big) \,ds + \big(W_1( \hat{\mu}^{n,k}_T, m_T) \wedge 1\big) \Big]
    \end{equation*}
	where $\hat{\mu}^{n,k}_s := \frac{1}{n} \sum_{k' \in \hat{K}^n_s\setminus\{k\}} \delta_{\hat{X}^{k'}_s}$ and $C=L \vee 2\|f\| \vee 2\|g\|$.
	We aim to show that the r.h.s. vanishes uniformly w.r.t. $k$ as $n$ goes to infinity.  
	We first observe that 
	\begin{equation}\label{eq:nash-subproof}
	  W_1 (\hat{\mu}^{n,k}_s, m_s) \leq W_1 (\hat{\mu}^{n,k}_s, \hat{\nu}^n_s) + W_1 (\hat{\nu}^n_s, m_s),
	\end{equation}
	where $ \hat{\nu}^n_s:=\frac{1}{n} \sum_{k' \in \hat{K}^n_s} \delta_{\hat{X}^{k'}_s}$. For the first term on the r.h.s. of~\eqref{eq:nash-subproof}, we use the duality result of Lemma~\ref{lem:duality}  to obtain that
	\begin{equation*}		
		W_1 (\hat{\mu}^{n,k}_s, \hat{\nu}^n_s)
		= \sup_{\varphi \in \mathrm{Lip}^0_1(\R^d)} \Big\{ \hat{\mu}^{n,k}_s(\varphi) - \hat{\nu}^n_s(\varphi)\Big\}
		+ \Big|\hat{\mu}^{n,k}_s(\R^d) - \hat{\nu}^n_s(\R^d)\Big|
	    \le \frac{1}{n} \Big(1 + \big|\hat{X}^{k}_s\big| \Big),
	\end{equation*}	
    where $\mathrm{Lip}^0_1(\R^d)$ stands for the collection of all functions $\varphi: \R^d \to \R$ with Lipschitz constant smaller or equal to $1$ and such that $\varphi(0) = 0$. 
	Hence, we deduce that
	\begin{equation*}
	 \E\Big[W_1 (\hat{\mu}^{n,k}_s, \hat{\nu}^n_s)\Big]\leq \frac{C}{n} \big(1 + \E\big[|X_0|\big]\big),
	\end{equation*}
	where $C$ depends solely on $T$ and $\|Du\|$.
	As for the second term on the r.h.s. of~\eqref{eq:nash-subproof}, it follows from the law of large number that for any $\varphi: \R^d \to \R$ continuous such that $|\varphi(x) | \le C(1+|x|)$,
	\begin{equation*}
		\hat{\nu}^{n}_s (\varphi)
		=
		\frac{1}{n} \sum_{i=1}^n \sum_{k' \in \hat{K}^{n,i}_s}  \varphi(\hat{X}^{k'}_s)
		\xrightarrow[n\to\infty]{a.s.}
		\E\bigg[\sum_{k' \in \hat{K}^1_s}  \varphi(\hat{X}^{k'}_s)\bigg] = m_s(\varphi),
	\end{equation*}
	where $\hat{K}^{n,i}_s:=\{k\in \hat{K}^n_s;\ k\succeq i\}$.
	In view of Lemma~\ref{lemm:Wp_App}, it is equivalent to
	\begin{equation*}
		W_1( \hat{\nu}^{n}_s, m_s) 
		\xrightarrow[n\to\infty]{a.s.} 0.
	\end{equation*}
	Thus we conclude by the dominated convergence theorem that
	\begin{equation*}
		\sup_{k\in\K} ~ \E \Big[ \int_0^T \big( W_1 (\hat{\mu}^{n,k}_s, m_s) \wedge 1 \big) ds + \big(W_1( \hat{\mu}^{n,k}_T, m_T) \wedge 1\big) \Big]
		\xrightarrow[n\to\infty]{} 0.
	\end{equation*}
	
	\noindent $\mathrm{(iii)}$ The third term on the r.h.s. of~\eqref{eq:nash-proof} can be treated exactly like the second one.\qed
	\end{proof}

\subsection{Numerical Example}
\label{subsec:numeric}

It is often difficult to solve a general MFG system, while it is possible to find explicit solutions in special cases, such as the linear-quadratic models. See, \eg, Bardi~\cite{Bar2012} and Carmona~\textit{et al.}~\cite{lachapelle2013}. In order to illustrate the behaviour of the equilibrium of the MFG with branching, we study the following simple linear-quadratic model:
\begin{eqnarray}
 \partial_t u + \Delta u - \frac{1}{2}\left|Du\right|^2 - \gamma u = 0, & &  \text{in }[0,T)\x\R, \label{eq:hjb-ex}\\
 \partial_t m - \Delta m - \mathrm{div}\left(m Du\right) -  \lambda x^2 m = 0, & & \text{in }(0,T]\x\R, \label{eq:fk-ex}\\
  u(T,\cdot) = g\left(\cdot,m(T)\right), ~ m(0) = m_0, & & \text{in }\R, \label{eq:bc-ex}
\end{eqnarray}
where $\gamma>0$, $\lambda\geq 0$, $\delta >0$, $x_0\in\R$ and 
\begin{equation*}
 g\left(x,\mu\right) := \frac{1}{2}\left(x-x_0\right)^2 +\frac{\d}{2}\Big(x-\frac{1}{\mu(\R)}\int_{\R} {y\, \mu(dy)}\Big)^2.
\end{equation*}
In other words, we take 
\begin{equation*}
 f(x,\mu)=0 \quad \text{and} \quad \sum_{\ell\in\N}{\ell p_\ell(x)} = 1 + \frac{\lambda}{\gamma} x^2.
\end{equation*}
Notice that we are not in the exact setting of the previous sections as $\sum_{\ell\in\N}{\ell p_\ell}$ is not bounded. This is not a problem as we can construct an explicit solution to the MFG above.

\begin{Remark}
\rm
(i) The terminal cost $g$ indicates that particles aim to reach the desired position $x_0$ while being close to the average position of all living particles.

\no (ii) The quadratic form of $\sum_{\ell\in\N}{\ell p_\ell}$ implies that particles further away from the origin generate more particles. In particular, it is position-dependent so the problem does not reduce to classical MFG by a simple change of variable as explained in Remark~\ref{rem:uniqueness_edp}.
\end{Remark}

\begin{Proposition}
Assume that $m_0$ is a Gaussian distribution $\cN(\rho_0,v_0)$. Then there exists an equilibrium for the MFG with branching~\eqref{eq:hjb-ex}--\eqref{eq:bc-ex} if the horizon is short enough.
\end{Proposition}

\begin{proof}
Given an environment measure $\mu=(\mu_t)_{t\in[0,T]}$, we can solve the HJB equation using the standard argument for the linear-quadratic model so that
\begin{equation*}
 Du(t,x) = a_t x + b_t,
\end{equation*}
where 
\begin{equation}\label{eq:lq-a} 
	a_t  :=  \frac{\g\left(1+\d\right)}{\left(1+\d+\g\right)e^{\g (T-t)} - \left(1+\d\right)}, 
 	~~~~
	b_t := -\left(x_0 + \d\int_{\R} {y \,\mu_T(dy)}\right)e^{-\int_t^T {\left(a_s+\g\right) ds}}. 
\end{equation}
Notice that the solution to the HJB equation is coupled with the solution to the Fokker--Plank equation only through the term $\int_\R {y \,\mu_T(dy)}$ which appears in the coefficient~$b$.

Next we study the Fokker--Plank equation. We observe that the normalized density $\bar m := \frac{m}{m(\R)}$ formally satisfies 
\bea\label{ex-barm}
 \partial_t \bar m - \Delta \bar m - \mathrm{div}(\bar m Du) + \l\left(\int_\R  y^2 \,\bar m(dy) - x^2\right) \bar m = 0, & \text{in }(0,T]\x\R.
\eea
We claim that there is a solution of the form $t\mapsto \bar m(t)$ where $\bar m(t)$ is a Gaussian distribution $\Nc(\rho_t,v_t)$.
A straightforward computation yields that PDE~\eqref{ex-barm} has a Gaussian solution if and only if the following system admits a solution:
\bea\label{riccatisys}
\begin{cases}
\dot v = 2\l v^2 -2 a v + 2, \\
\dot \rho = (2\l v-a)\rho - b,
\end{cases}
\q (v_0,\rho_0)~~\mbox{given.}
\eea
The Picard--Lindel\"of theorem ensures that the Riccati equation satisfied by $v$ has a unique solution if the horizon is short enough. Once $v$ is given, we can solve the linear equation satisfied by $\rho$ so that
\bea\label{solmu}
\rho_t =  \rho_0 e^{\int_0^t {(2\l v_s - a_s)\,ds}} - \int_0^t  {e^{\int_s^t {(2\l v_r -a_r)\,dr}} b_s \,ds}.
\eea

Based on the analysis above, we may construct a solution to the MFG as follows. First, $a$ and $v$ can be calculated through~\eqref{eq:lq-a} and \eqref{riccatisys} as they only depend on known coefficients. Then, using the expression of $b$ in~\eqref{eq:lq-a}, we can solve the fixed point condition and determine $\rho_T$ through~\eqref{solmu} so that
\bea\label{muT}
\rho_T = \frac{\rho_0 \hat\theta + x_0 \theta }{1-\d \theta},\quad\text{where }
\theta := \int_0^T {e^{\int_t^T {(2\l v_s - 2 a_s - \g)\,ds}} \,dt},
\ \hat\theta := e^{\int_0^T {(2\l v_t -a_t)\,dt}}.
\eea
Finally, we obtain $b$ and $\rho$ through \eqref{eq:lq-a} and \eqref{solmu}.
\qed
\end{proof}

From the previous proof, it turns out that a solution to the MFG with branching~\eqref{eq:hjb-ex}--\eqref{eq:bc-ex} can be calculated numerically. Let us present the result of a numerical test where $T=1,$ $\gamma = 0.2,$ $\delta =0.5,$ $\rho_0=0,$ $v_0=1$ and $x_0 = 5$. We shall focus on the effect of branching on the MFG equilibrium, which is illustrated by changing the value of the parameter $\l$. Note that if $\l$ becomes too large, the solution to the Riccati equation in~\eqref{riccatisys} explodes. 

\begin{figure}[h]
\centering
\begin{tabular}{|c|c|}
\hline
\subf{\includegraphics[width=60mm]{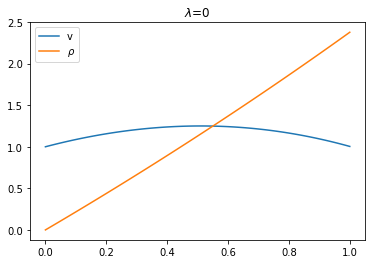}}
     {}
&
\subf{\includegraphics[width=60mm]{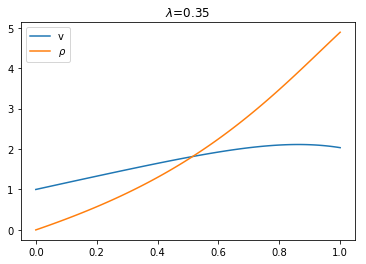}}
     {}
\\
\hline
\subf{\includegraphics[width=60mm]{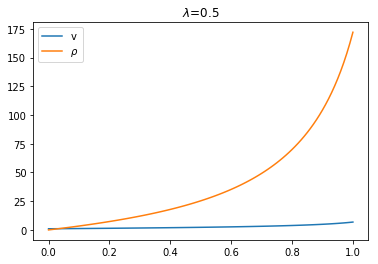}}
     {}
&
\subf{\includegraphics[width=60mm]{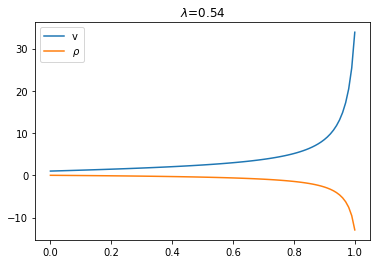}}
     {}
\\
\hline
\end{tabular}
\caption{The change of the equilibriums for different $\l.$}\label{fig:ex}
\end{figure}

In view of Figure~\ref{fig:ex}, when there is no branching, \ie, $\l=0$, the mean $\rho$ approaches the desired position $x_0=5$ almost linearly, and the variance $v$ is controlled to be small when the time approaches the maturity. Once the branching appears, the mean moves more quickly towards $x_0$ as illustrated in the case $\lambda=0.35.$ As we increase further the parameter $\lambda$,  it eventually crosses the desired position and keeps growing larger as particles further away from the origin generate more particles.  Meanwhile the variance $v$ becomes slightly larger but remains well-controlled. A singularity appears for $\l$ close to $0.5$ as $\rho_T$ becomes infinite since $\d\theta =1$ in \eqref{muT}. The unexpected case happens right after the singularity when $\d\theta >1$ in \eqref{muT}. As a result, the mean $\rho$ moves in the negative direction, moving away from the desired position $x_0$ instead of approaching it. Meanwhile the variance grows larger and larger, and eventually explodes for $\l$ close to $0.54.$

\section{Probabilistic Approach}
\label{sec:probabilistic}

	In this section, we provide a weak probabilistic formulation of the MFG with branching described in Section~\ref{subsec:MFG_descrip}. We follow the inspiration of Lacker~\cite{lacker2015} for classical MFG and use the relaxed formulation of stochastic control problem introduced by El Karoui~\textit{et al.}~\cite{elkaroui1987}. 
	First we focus on the relaxed control problem for diffusions solved by each agent in Section~\ref{sec:agent-problem}. 
	Then we formulate and study the corresponding relaxed control problem for branching diffusions in Section~\ref{sec:system-problem}. 
	This allows us to define a relaxed MFG with branching in Section~\ref{sec:bmfg} and to ensure existence of solutions in Theorem~\ref{thm:main_result}.
	Some technical proofs are completed in Section~\ref{sec:proof}.

\subsection{Relaxed Control of Diffusions}
\label{sec:agent-problem}

	Following~\cite{elkaroui1987}, we introduce a relaxed control problem for diffusion processes which corresponds to the optimization problem solved by each agent in the game.
	The main idea is to consider a controlled martingale problem on an appropriate canonical space corresponding to the pair formed by the control and the diffusion.
	
 Let us introduce first the canonical space $\Om := \cV\x \Cc \x \Qc$ where
	\begin{itemize}
		\item $\cV$ is the space of measures $\lambda$ on $[0,T] \x A$ with first marginal corresponding to the Lebesgue measure, equipped with the weak topology;
		\item $\Cc$ is the space of continuous maps $x:[0,T]\to\R^d$, equipped with the uniform topology;
		\item $\Qc$ is the space of locally finite integer--valued  measures $q$ on $[0,T] \x \R_+$, equipped with the vague topology.
	\end{itemize}
	Denote by $\Lambda$, $X$ and $Q$ the canonical projection from $\Om$ onto $\cV$, $\Cc$ and $\Qc$ respectively. Then we define the canonical filtration  $\F = (\Fc_s)_{s\in [0,T]}$ as
	\begin{equation*}
		\Fc_s := \sigma \big(\1_{[0,s]} \Lambda, X_{s\wedge\cdot}, \1_{[0,s]} Q \big).
	\end{equation*}
	We can further define a $\F$-predictable process $(\Lambda_s)_{s\in [0,T]}$ valued in $\Pc(A)$ such that $\Lambda(ds, da) = \Lambda_s(da)\, ds$, see~\cite[Lemma 3.2]{lacker2015}.
	
	Next we introduce the notion of relaxed control as solution to a controlled martingale problem.
    Recall that the operator $\Lc^{\mu, a}_s$ in~\eqref{eq:Lc} is defined as
    \begin{equation*}
   \Lc^{\mu,a}_s {\varphi}(x) = \frac{1}{2}\mathrm{tr}\left(\sigma\sigma^*\left(s,x,\mu_s,a\right) D^2{\varphi}\left(x\right)\right)  + b\left(s,x,\mu_s,a\right) \cdot D{\varphi}\left(x\right).
\end{equation*}

 \begin{Definition} \label{def:single_relax_ctrl}
  	Given $(t,x) \in [0,T] \x \R^d$ and $\mu=(\mu_t)_{t\in[0,T]}, \mu_t\in\Mc(\R^d),$ an element of $\Rc(t,x,\mu)$ is a probability measure $\P$ on $\Om$ such that
  	\begin{itemize}
  	 \item[\rm(i)] $\P(X_s =x,\, s\leq t) = 1$;
  	 
  	 \item[\rm(ii)] for all $\varphi \in C^2_b(\R^d)$, the process 
		\begin{equation*}
			M^{t,\mu,\varphi}_s
			:= 
			\varphi(X_s) - \int_t^s \!\! \int_A 
			\Lc_r^{\mu,a}\varphi(X_r)
			\,\Lambda(dr,da),
		\end{equation*}
	  	is a $(\P, \F)$--martingale
  		on $[t,T]$;

		\item[\rm(iii)] $Q(ds,dz)$ is a $(\P, \F)$--Poisson random measure with intensity $\mathbf{1}_{(t,T]}(s)\,ds\, dz.$
  	\end{itemize}
 \end{Definition}
 
	\begin{Remark}	
	\rm
	  Let $\P \in \Rc(t,x,\mu)$ be such that $\Lambda(ds, da) = ds\, \delta_{\alpha_s}(da)$ for some predictable process $\alpha$.
		Then there exists a Brownian motion $B$ on a possibly enlarged space such that $X$ satisfies
		\begin{equation*}
			dX_s = b(s, X_s, \mu_s, \alpha_s) \,ds + \sigma(s, X_s, \mu_s, \alpha_s) \,dB_s\quad
			\text{on }[t,T].
		\end{equation*}
		In particular, the notion of relaxed control generalizes the classical notion of control process as  
		the disintegration $\lambda(ds,da)=ds\, \lambda_s(da)$ may not be supported on the Dirac measures.
	\end{Remark} 
	
   Let us now consider the relaxed control problem corresponding to the following value function:
    \begin{equation}\label{eq:valuefunction}
		v(t, x, \mu) 
		:= 
		\inf_{\P \in \Rc(t,x,\mu)}
		J(t, x, \mu, \P),
	\end{equation}
	with the cost function
	\begin{equation*}
		J(t, x, \mu, \P)
		:=
		\E^{\P}
		\Big[
			\int_t^{\tau}\!\!\! \int_A f(s, X_s, \mu_s, a) \,\Lambda(ds,da)+ g(X_T,\mu_T) \1_{\tau = T}
		\Big],
	\end{equation*}
	and the stopping time
	\begin{equation*}
		\tau
		:=
		\inf \Big\{ 
			s > 0;\  Q \Big(\{s\} \x \Big[0, \int_A \gamma(s, X_s, \mu_s, a) \Lambda_s(da) \Big] \Big)= 1
		\Big\}\wedge T.
	\end{equation*}  			
	Then we define $\Rc^*(t,x, \mu)$ the collection of all optimal relaxed controls as
	\begin{equation*}
	 \Rc^*(t,x,\mu) := \big\{ \P \in \Rc(t,x,\mu);\ J(t,x,\mu,\P) = v(t,x,\mu) \big\}.
	\end{equation*}

	One of the main advantages of the relaxed control formulation is that it allows us to derive existence of optimal controls under general conditions.
	The next proposition supports this statement.

	\begin{Assumption} \label{assum:RelaxedCtrl}
      $b$, $\sigma$, $\gamma$, $f$ and $g$ are bounded, continuous w.r.t. $(x,\mu,a)\in\R^d\x \Mc(\R^d)\x A$.
    \end{Assumption}
		
	\begin{Proposition} \label{prop:non-empty-diffusion}
		Under Assumption~\ref{assum:RelaxedCtrl}, the set $\Rc^*(t,x,\mu)$ is nonempty. 
	\end{Proposition}

	 The existence of an optimal relaxed control follows from the compactness of $\Rc(t,x,\mu)$ and the continuity of the cost function w.r.t. $\P\in\Rc(t,x,\mu)$.
	We refer to \cite{elkaroui1987, lacker2015} for a detailed proof in a slightly different context.\footnote{The only difficulty in our setting comes from the stopping time $\tau$ which might impair at first sight the continuity of the cost function.
	Nevertheless, $\tau$ is continuous on $\Om$ except on the set
	\begin{equation*}
		 \Om_0 := \Big\{
			\om = (\lambda, x, q) \in \Om;\
			\exists s \in [0,T],\ 
			q \Big( \{s\} \x \Big\{\int_A \gamma(s, x_s, \mu_s, a) \lambda_s(da) \Big\}\Big) \geq 1
		\Big\},
	\end{equation*}
	and $\P(\Om_0) = 0$ whenever $Q$ is $(\P,\F)$--Poisson random measure with intensity $\mathbf{1}_{(t,T]}(s)\, ds\, dz.$}

	\begin{Remark}
	\label{rem:Markov_Selection}
	\rm
	If we assume further that for all $(s,x,\mu)\in [0,T]\x\R^d\x\Mc(\R^d),$
	\begin{equation*}
	 \big\{\big(b(s,x,\mu,a),\sigma\sigma^*(s,x,\mu,a),\gamma(s,x,\mu,a),z\big);\ a\in A,\, z\leq f(s,x,\mu,a)\big\}
	\end{equation*}
	is convex,
	then it is well-known that there exists an optimal strict Markov control, \ie,   an element $\P \in \Rc^*(t,x,\mu)$ such that $\Lambda(ds, da) = ds \,\delta_{\alpha(s,X_s)}(da)$ for some map $\alpha: [0,T] \x \R^d \to A$. 
	We refer to~\cite{elkaroui1987,lacker2015} for more details.
	\end{Remark}

\subsection{Relaxed Control of Branching Diffusions}
\label{sec:system-problem}
  By extending the ideas of~\cite{elkaroui1987}, we can formulate a relaxed control problem for branching diffusion processes where every agent aims at solving the problem of Section~\ref{sec:agent-problem}. 
  Although the appropriate formulation readily follows from Proposition~\ref{prop:dsm}, the problem of existence of optimal solutions  raises significant difficulties as the optimization criterion is rather peculiar, see Remark~\ref{rem:difficulty}.

	Recall that  $\bar{A}$ is the collection of all sequences $\bar{a}=(a^k)_{k\in\K}, a^k\in A.$ Let us introduce first the canonical space $\bar{\Om}:=\bar{\cV}\x \cD$ where
	\begin{itemize}
		\item $\bar{\Vc}$ is the space of measures $\bar{\lambda}$ on $[0,T]\x\bar{A}$ with first marginal corresponding to the Lebesgue measure and such that
		\begin{equation*}
			 \bar{\lambda}(dt,d\bar{a})=dt\, \bigotimes_{k\in\K} \lambda^k_t(da^k),
		 \end{equation*}
		where $\lambda^k(dt, da)=dt\,\lambda^k_t(da)$ is the pushforward of $\bar{\lambda}(dt,d\bar{a})$ by $\pi^k:(t,\bar{a})\mapsto (t,a^k)$, endowed with the weak topology; 
 
		\item $\cD$ is the space of all c\`{a}dl\`{a}g paths $z:[0,T]\to E$, endowed with the Skorokhod topology.
	\end{itemize}
	Let $\bar{\Lambda}$ and $Z$ be the canonical projection from $\bar{\Om}$ onto $\bar{\Vc}$ and $\Dc$ respectively. Then the canonical filtration $\bar{\F} = (\bar{\Fc}_t)_{t\in [0,T]}$ is given by
	\begin{equation*}
		\bar{\cF}_t := \sigma\left(\mathbf{1}_{[0,t]}\bar{\Lambda}, Z_{t\wedge\cdot}\right).
	\end{equation*}
	
	We also introduce the time of birth $S_k,$ the time of death $T_k,$ the position $X^k$ and the control $\Lambda^k$ of particle $k\in\K$ as follows:
	\begin{align*}
		 & S_k := \inf{\left\{t\geq 0;\, \langle Z_t, \mathbf{1}_{\{k\}}\rangle = 1\right\}}\wedge T,& & T_k := \inf{\left\{t> S_k;\, \langle Z_t, \mathbf{1}_{\{k\}}\rangle = 0\right\}}\wedge T,\\
		& X^k_t := \langle Z_t, \mathbf{1}_{\{k\}} \mathrm{Id}_{\R^d}\rangle \quad \forall\, t\in[S_k,T_k], & & \Lambda^k := \bar{\Lambda}\circ(\pi^k)^{-1}.
	\end{align*}
	Then the set $K_t$ of all particles alive at time $t\in[0,T]$ and the size $N_t$ of the population are given by
	\begin{equation*}\label{def:NK}
		K_t := \left\{k\in\K;\, \langle Z_t, \mathbf{1}_{\{k\}}\rangle = 1\right\},\quad N_t := \langle Z_t, \mathbf{1}\rangle = \# K_t.
	\end{equation*} 	
	 
    Next we introduce the notion of relaxed control for branching diffusion processes as solution to a controlled martingale problem deriving from Proposition~\ref{prop:dsm}. 
    Recall that the operator $\Hc^{\bar{a}}$ in~\eqref{eq:bd-generator} is defined as
\begin{multline*}
 \Hc^{\mu,\bar{a}}_s {\Phi_{\bar{\varphi}}} \left(e\right) := \frac{1}{2} \Phi_{\bar{\varphi}}''\left(e\right) \sum_{k\in K} {\big|D {\varphi^k} (x^k) \sigma(s,x^k, \mu_s, a^k) \big|^2} + \Phi_{\bar{\varphi}}'\left(e\right) \sum_{k\in K} {\Lc^{\mu,a^k}_s {\varphi^k}(x^k)} \\
 + \sum_{k\in K} { \gamma(s,x^k, \mu_s, a^k) \left(\sum_{\ell\in\N} { \Phi_{\bar{\varphi}}\Big(e - \delta_{(k,x^k)}  +\sum_{i=1}^{\ell} \delta_{(ki,x^k)} \Big) p_\ell(s,x^k,\mu_s)  - \Phi_{\bar{\varphi}}\left(e\right) } \right) }.
\end{multline*}
	
\begin{Definition} \label{def:tree}
	Given $\mu=(\mu_t)_{t\in[0,T]}, \mu_t\in\Mc(\R^d),$ an element of $\Tc(\mu)$ is a probability measure $\Pb$ on $\bar{\Om}$ such that
\begin{itemize}
 \item[\rm(i)] $\Pb\circ Z_0^{-1} = m_0 \circ \pi^{-1}$ where $\pi:x \in \R^d\mapsto \delta_{(1,x)}\in E$;
 
 \item[\rm(ii)] for all $\Phi\in\cC^2_b(\R)$, $\bar{\varphi}\in\cC^2_b(\K\x\R^d)$, the process
\begin{equation*}
 M_t^{\mu,\Phi_{\bar{\varphi}}} := \Phi_{\bar{\varphi}}\left(Z_t\right) - \int_0^t \!\! \int_{\bar{A}} {\cH^{\mu,\bar{a}}_s {\Phi_{\bar{\varphi}}}\left(Z_s\right) \,\bar{\Lambda}(ds, d\bar{a})},
\end{equation*}
 is a $(\Pb,\bar{\F})$--martingale on $[0,T].$

\end{itemize}

\end{Definition}

\begin{Remark} \label{rem:MartingalePb}
	\rm The classical semimartingale theory provides an equivalent formulation of Condition~(ii) above, namely, for all $\bar{\varphi} \in C^2_b(\K\x\R^d),$
	the process $Y_t = \sum_{k \in K_t} \varphi^k(X^k_t)$ is a $\Pb$--semimartingale with characteristics $(\bar{B}, \bar{C}, \bar{\nu})$ given by
	\begin{equation*}
	 \bar{B}_t=\sum_{k\in K_t} {B^k_t}, \quad \bar{C}_t=\sum_{k\in K_t} {C^k_t}, \quad \bar{\nu}(dt, dy) =\sum_{k\in K_t} {\nu^k(dt, dy)},
	\end{equation*}
	where
	\begin{align*}
		dB^k_t &:= \int_A {\Lc^{\mu,a}_t \varphi^k(X^k_t) \,\Lambda^k(dt, da)} + \int_{|y| \le 1} {y \,\nu^k(dy, dt)}, \\
		dC^k_t &:= \int_A {\big|D\varphi^k(X^k_t) \sigma(t,X_t^k,\mu_t,a)\big|^2 \,\Lambda^k(dt, da)} \\
		\nu^k(dt, dy) &:= \int_A \gamma(t,X^k_t, \mu_t, a) \sum_{\ell\in\N} p_{\ell}(t,X^k_t,\mu_t) \,\delta_{\big\{\sum_{i=1}^{\ell} \varphi^{ki}(X^k_t) - \varphi^k(X^k_t)\big\}}(dy) \Lambda^k(dt, da).
	\end{align*}
	See, \eg, Jacod and Shiryaev~\cite[Theorem II.2.42]{jacod03}.
\end{Remark}	

	The set of optimal relaxed controls $\cT^*(\mu)$ is then defined as
	\begin{equation} \label{eq:def_T_star}
		\Tc^*(\mu) := \Big\{ \Pb \in \Tc(\mu);\ J_k\left(\mu, \Pb\right) = \E^{\Pb} \big[v\big(S_{k}, X^k_{S_k},\mu\big) \big],\ \forall\,k \in \K \Big\},
	\end{equation}
	with the value function $v$ defined in~\eqref{eq:valuefunction} and the cost function for particle $k$ given as
	\begin{equation*}
		J_k\left(\mu, \Pb\right) := \E^{\Pb}\Big[\int_{S_k}^{T_k} \!\! \int_A {f\big(s, X^k_s, \mu_s, a\big) \,\Lambda^k(ds,da)} + g\big(X^k_T,\mu_T\big) \mathbf{1}_{T_k= T}\Big].
	\end{equation*}
    This corresponds to a branching diffusion process where every agent minimizes its own cost criterion as described in Section~\ref{subsec:MFG_descrip}. 
    
   As expected for a relaxed formulation, we can ensure  existence of optimal controls under rather general conditions. However the proof is fairly delicate in this setting and we postpone it to Section~\ref{sec:proof-empty}.

  \begin{Assumption}\label{ass:relaxed-branching}   
  \rmi  There exists $C>0$ such that for all $t\in [0,T]$, $x,y\in\R^d$, $\mu\in\Mc(\R^d)$, $a\in A$,
		 \begin{equation*}
			\big| b(t,x,\mu,a) - b(t,y,\mu,a) \big| + \big| \sigma(t,x,\mu,a) - \sigma(t,y,\mu,a) \big| \leq C \left| x - y \right|.
		\end{equation*}  
 \rmii $(p_\ell)_{\ell\in\N}$ and $\sum_{\ell\in\N} {\ell p_\ell}$ are continuous w.r.t. $(x,\mu)\in\R^d\x\Mc(\R^d)$ and $\sum_{\ell\in\N} {\ell^2 p_\ell}$ is bounded.
  \end{Assumption}
  
	\begin{Proposition} \label{prop:non-empty-system}
		Let Assumptions~\ref{assum:RelaxedCtrl} and \ref{ass:relaxed-branching} hold. Then
		the set $\Tc^*(\mu)$ is nonempty.
	\end{Proposition}
	
	\begin{Remark}\label{rem:difficulty}
	\rm
		The difficulty to derive existence of an element in $\Tc^*(\mu)$ comes from the fact that each particle minimizes its own cost function rather than a cost function for the whole population. To fix ideas, consider a relaxed control problem of the form 
	  \begin{equation*}
	   \inf_{\bar{\P}\in\Tc(\mu)} {\bar{J}\left(\bar{\P}\right)} \quad \text{where }\bar{J}\left(\bar{\P}\right):=\E^{\Pb}\left[\bar{g}(Z)\right],
	  \end{equation*}
	  for some map $\bar{g}: \Dc\to\R.$ 
	 Then existence of optimal controls should follow by extension of the arguments in~\cite{elkaroui1987} showing that $\Tc(\mu)$ is compact and $\bar{J}$ is continuous.
	  However, our problem does not belong to this class and another strategy is needed to ensure existence of an optimal solution.  
	\end{Remark}
	
	\begin{Remark}
	\label{rem:Markov_Selection_Tree}
	\rm
	 The proof of Proposition~\ref{prop:non-empty-system} relies on concatenation of optimal particles in $\Rc^*(t,x,\mu)$ to construct an optimal tree in $\Tc^*(\mu).$
	 In view of Remark~\ref{rem:Markov_Selection}, under additional convexity assumptions, we can ensure existence  of an optimal control $\Pb\in\Tc^*(\mu)$ such that every particle $k$ runs an optimal control of the form $\Lambda^k(dt, da) = dt\, \delta_{\alpha(t, X^k_t)}(da)$ for some map $\alpha: [0, T] \x\R^d \to A.$
	\end{Remark}

\subsection{Relaxed MFG with Branching}
\label{sec:bmfg}

 We are now in a position to introduce the relaxed formulation of the MFG with branching described in Section~\ref{subsec:MFG_descrip} and to establish existence of solutions.

\begin{Definition} \label{def:MFG_proba}
	A solution to the relaxed MFG with branching is a probability distribution $\bar{\mu}\in\Pc(\Dc)$ such that $\bar{\mu}\in\{\Pb\circ Z^{-1};\, \Pb\in\Tc^*(\mu)\}$ where $\mu$ is defined from $\bar{\mu}$ as, for  all $\varphi:\R^d\to\R$ bounded,
	\begin{equation} \label{eq:def_mut}
	   \int_{\R^d} \varphi(x) \,\mu_t(dx) :=   \int_{E} \mathrm{Id}_{\varphi}(e) \,\bar{\mu}_t(de) = \E^{\bar{\mu}} \Big[ \sum_{k\in K_t} {\varphi(X^k_t)} \Big].
	\end{equation}
\end{Definition}

  The main result of this section ensures existence of solutions to the relaxed MFG with branching under rather general conditions.

\begin{Theorem}\label{thm:main_result}
	Let Assumptions~\ref{assum:RelaxedCtrl} and \ref{ass:relaxed-branching} hold. Then there exists a solution to the relaxed MFG with branching. 
\end{Theorem}

\begin{proof} The idea of the proof is to apply  Kakutani's fixed point theorem~\cite[Corollary~17.55]{aliprantis2006} to the set-valued map
	\begin{equation*}
		\psi: \bar{\mu}\in\Pc(\Dc) \mapsto \left\{\Pb\circ Z^{-1};\, \Pb\in\Tc^*(\mu)\right\} \subset\Pc(\Dc).
	\end{equation*}
 	To this end, we need to check that the range of $\psi$ is contained in a compact convex subset of $\Pc(\Dc)$ and that $\psi$ has a closed graph and non-empty, convex values.
 	For all $\bar{\mu}\in\Pc(\Dc),$ the set $\psi(\bar{\mu})$ is nonempty in view of Proposition~\ref{prop:non-empty-system}
	and is a convex subset of $\Pc(\Dc)$ by definition of $\Tc^*(\mu)$ in~\eqref{eq:def_T_star}. 
	The conclusion then follows by applying Lemma~\ref{lem:closed-graph} and Lemma~\ref{lem:psi} below. 
	Their proofs are postponed to Section~\ref{sec:closed-graph} and Section~\ref{sec:proof-compact} respectively. 
	\qed
\end{proof}

	\begin{Lemma}\label{lem:closed-graph}
		Let Assumptions~\ref{assum:RelaxedCtrl} and \ref{ass:relaxed-branching} hold. 
		Then the set--valued map $\bar{\mu}\mapsto\Tc^*(\mu)$ has a closed graph, i.e., $\{(\bar{\mu}, \Pb);\ \Pb \in \Tc^*(\mu) \}$ is a closed subset of $\Pc(\Dc) \x \Pc(\Omb)$.
	\end{Lemma}

	\begin{Lemma}\label{lem:psi}
		Let Assumptions~\ref{assum:RelaxedCtrl} and \ref{ass:relaxed-branching} hold.
		Then the set $\bigcup_{\bar{\mu}\in\Pc(\Dc)}\Tc(\mu)$ is contained in a compact convex subset of $\Pc(\bar{\Om})$.
	\end{Lemma}

 \begin{Remark}
 \rm
  \rmi The ideas of this section can easily be extended to a broader class of branching diffusion processes. 
  For instance, we can take into account an immigration phenomenon to allow new players to enter the game exogenously at any time. 
  We can also consider that child particles start from a different position than the mother particle, say for instance to allow players to share their wealth between substitude players.
  
  \no\rmii The analysis of this section readily extends to a wide variety of objective functions so long as the cost function is continuous and convex.
 
 \end{Remark}
 
 \begin{Remark}
 \rm
 In view of Remark~\ref{rem:Markov_Selection_Tree}, under additional convexity assumptions, we can construct on some filtered probability space a branching diffusion process  where each particle runs an identical Markov control that solves the MFG with branching described in Section~\ref{subsec:MFG_descrip}.
  The details are omitted for the sake of conciseness.
 \end{Remark}

	\begin{Remark} \label{rem:uniqueness_proba}\rm
		In lines with Remark~\ref{rem:difficulty}, as each agent optimizes its own (local) cost function rather than the (global) cost of the whole population,
		it leads to serious difficulties in order to extend the variational arguments used in classical MFG.
		Concretely, given $\Pb \in \Tc^*(\mu)$ and $\Pb' \in \Tc(\mu) \setminus \Tc^*(\mu)$, the birth time $S_k$
		of particle $k\neq 1$ may have different distributions under $\Pb$ and $\Pb'$,
		and thus it may not be true that $J_k(\mu, \Pb) \le J_k(\mu, \Pb')$.
		In particular, because of the failure of this kind of variational argument,
		it is not straightforward to extend the arguments in \cite{CarLac2015} to prove uniqueness,
		or those in \cite{lacker2016general} to prove convergence  of solutions to the $n$-player game toward the  MFG.
		These delicate questions are left for future research.
	\end{Remark}

\subsection{Technical Proofs}
\label{sec:proof}

	We now provide the proofs of Lemma~\ref {lem:closed-graph}, Lemma~\ref{lem:psi} and Proposition~\ref{prop:non-empty-system}.
	Let us assume that Assumptions~\ref{assum:RelaxedCtrl} and \ref{ass:relaxed-branching} hold throughout this section.
	
\subsubsection{Proof of Lemma \ref{lem:closed-graph}}
\label{sec:closed-graph}

	Recall that $N_t=\langle Z_t,\mathbf{1}\rangle$ denotes the size of the population at time $t\in[0,T].$ 
	For every $\mu=(\mu_t)_{t\in[0,T]},$ $ \mu_t\in\Mc(\R^d)$, we denote by $\Tc_{\mathrm{loc}} (\mu)$ the collection of all probability measures $\Pb$ on $\Omb$
	such that $N_0 = 1$, $\Pb$--a.s., and the process
	\begin{equation*}
		\Phi(N_t) 
		- 
		\int_0^t \sum_{k \in K_s} \int_{A} \gamma(s,X^k_s, \mu_s, a) 
			\Big(
				\sum_{\ell\in\N} p_{\ell}(s,X^k_s,\mu_s) \Phi(N_s + \ell - 1) - \Phi(N_s)
			\Big)
		\Lambda^k(da, ds),
	\end{equation*}
	is a $\Pb$--local martingale  for all $\Phi \in \Cc^2_b(\R)$.
	Notice that $\Tc(\mu) \subset \Tc_{\mathrm{loc}}(\mu)$.

	\begin{Lemma} \label{lem:Nt}
		There exists a constant $C> 0$, such that for all $\mu=(\mu_t)_{t\in[0,T]}, \mu_t\in\Mc(\R^d),$
		\begin{equation} \label{eq:bound_Nt}
			\sup_{\Pb\in \Tc_{\mathrm{loc}}(\mu)}\E^{\Pb} \Big[\sup_{t\in [0,T]} \left\{N_t^2\right\} \Big] \le C.
		\end{equation}
	\end{Lemma}
	\begin{proof}
	Let $\Pb \in \Tc_{\mathrm{loc}}(\mu)$,
	it follows by Theorem II.2.42 of Jacod and Shiryaev~\cite{jacod03} that the process $N$ is a $\Pb$--semimartingale with known characteristics, see Remark~\ref{rem:MartingalePb}. 
	Together with Corollary~II.2.38 of~\cite{jacod03}, it yields the following representation: 
	\begin{equation*}
	 N_t = N_0 + \int_{[0,t]\x\N} { n \,\mu_N(ds, dn)},
	\end{equation*}
	where $\mu_N(ds,dn) = \sum_{r\in [0,T]}{\1_{\{\Delta N_r\neq 0\}}\delta_{(r,\Delta N_r)}}(ds,dn)$ is a random measure with compensator 
	\begin{equation*}
	\nu_N(ds, dn) :=  \sum_{k \in K_s} \int_A {\gamma(s,X^k_s, \mu_s, a)  \sum_{\ell\in\N} p_{\ell}(s,X^k_s,\mu_s) \,\delta_{\{\ell -1\}}(dn) \, \Lambda^k(ds, da).}
	\end{equation*}
	We can further use a representation theorem~\cite[Theorem II.7.4]{ikeda89} to ensure existence of a Poisson measure $Q(ds,dz,dk)$ on $[0,T]\x\R_+\x\K$ with intensity $ds\,dz\,\sum_{k\in\K}\delta_k(dk)$ on an extension of the canonical space, such that 
	\begin{equation*}
	 N_t = N_0 + \int_{[0,t]\x\R_+\x\K} {\1_{k\in K_s}\sum_{\ell\in\N} {(\ell-1)} \1_{z\in I^\mu_\ell(s,X^k_s,\Lambda^k_s)} \,Q(ds, dz, dk)},
	\end{equation*}
	where $(\Lambda^k_s)_{s\in [0,T]}$ is a $\F$-predictable process valued in $\Pc(A)$ such that $\Lambda^k(ds, da) = ds\,\Lambda^k_s(da)$ and for all $(s,x,\kappa)\in[0,T]\x\R^d\x\Pc(A),$
	\begin{equation*}
	 I^\mu_\ell(s,x,\kappa) := \Big[\int_A {\gamma(s,x,\mu_s, a) \,\kappa(da)} \sum_{i=0}^{\ell-1}{p_i(s,x,\mu_s)}, \int_A {\gamma(s,x,\mu_s, a) \,\kappa(da)} \sum_{i=0}^{\ell}{p_i}(s,x,\mu_s)\Big).
	\end{equation*}
	Since $\gamma$ and $\sum_{\ell\in\N} \ell^2 p_\ell$ are uniformly bounded by assumption, a classical computation as in Proposition~\ref{prop:def} yields \eqref{eq:bound_Nt}.
	\qed
	\end{proof}

	\noindent {\bf Proof of Lemma \ref{lem:closed-graph}.}
	Let $(\bar{\mu}_n,\Pb_n)_{n\in\N}$ be a sequence in $\Pc(\Dc)\x\Pc(\bar{\Om})$ satisfying $\Pb_n\in\Tc^*(\mu_n)$ and converging to $(\bar{\mu},\P).$ We aim to show that $\Pb\in\Tc^*(\mu).$
 
  \noindent\rmi First we observe that $\Pb\circ Z_0^{-1} = m_0\circ \pi^{-1}$ since the projection $Z_0$ is continuous for the Skorokhod topology. See, \eg, Billingsley~\cite[Theorem 12.5]{billingsley2013}.
 
	\noindent\rmii Next we check that $M^{\mu,\Phi_{\bar{\varphi}}}$ is a martingale under $\Pb$ for all $\Phi\in\cC_b^2(\R)$, $\bar{\varphi}\in\cC_b^2(\K\x\R^d)$. 
	Since $M^{\mu,\Phi_{\bar{\varphi}}}$ is a martingale under $\Pb_n$, it holds for all $t\leq s,$ $h:\Omb\to\R$ $\bar{\cF}_t$--measurable, continuous and bounded,
	\begin{equation}\label{eq:mart}
		\E^{\Pb_n}\Big[\Big(M^{\mu_n,\Phi_{\bar{\varphi}}}_{s} - M^{\mu_n,\Phi_{\bar{\varphi}}}_{t}\Big) h\Big] = 0.
	\end{equation}
	In view of Proposition~VI.3.14 in~\cite{jacod03}, there exists $D_{\Pb}\subset (0,T)$ countable such that
for all $t,s\notin D_{\Pb},$
   \begin{equation*}
    \E^{\Pb_n}\Big[\Big(\Phi_{\bar{\varphi}}(Z_s) - \Phi_{\bar{\varphi}}(Z_t)\Big) h\Big] \xrightarrow[n\to\infty]{} \E^{\Pb}\Big[\Big(\Phi_{\bar{\varphi}}(Z_s) - \Phi_{\bar{\varphi}}(Z_t)\Big) h\Big].
   \end{equation*}
     Additionally, it follows by a straighforward extension of Corollary~A.5 in Lacker~\cite{lacker2015} to the Skorokhod space that 
    \begin{equation*}
     (\bar{\lambda},z)\mapsto \int_t^s \!\! \int_{\bar{A}} {\cH^{\mu,\bar{a}}_r {\Phi_{\bar{\varphi}}}\left(z(r)\right) \,\bar{\lambda}(dr, d\bar{a})},
    \end{equation*}
	is continuous, and there exists a constant $C>0$ such that
	\begin{equation*}
		\Big| \int_t^s \!\!\int_{\bar{A}} {\cH^{\mu,\bar{a}}_r {\Phi_{\bar{\varphi}}}\left(Z_r\right) \,\bar{\Lambda}(dr, d\bar{a})} \Big|  \le C  \sup_{r\in [0,T]} \left\{N_r\right\}.
	\end{equation*}
	Using Lemma~\ref{lem:Nt}, we can thus pass to the limit in~\eqref{eq:mart} and deduce that for all $t,s\notin D_{\Pb},$
	\begin{equation*} 
		\E^{\Pb}\Big[\Big(M^{\mu,\Phi_{\bar{\varphi}}}_{s} - M^{\mu,\Phi_{\bar{\varphi}}}_{t} \Big) h\Big] = 0.
	\end{equation*}
     If $t$ or $s$ belong to $D_{\Pb},$ this equality still holds as we can pass to the limit with a decreasing sequence in $[0,T]\setminus D_{\Pb}$ converging to $t$ or $s$. We conclude that  $M^{\mu,\Phi_{\bar{\varphi}}}$ is a martingale under $\Pb.$

	\noindent\rmiii Finally, it holds for all $k\in \K$,
	\begin{equation*}
		J_k\left(\mu_n, \Pb_n\right)
		=
		\E^{\Pb_n}\big[v\big(S_{k}, X^k_{S_k},\mu_n \big) \big].
	\end{equation*}
	In view of Proposition~VI.2.7 in~\cite{jacod03}, the random variables $S_k$, $T_k$ and $X_{S_k}^k$ are continuous on $\Omb$. 
	Additionally, it follows by extension of the arguments in \cite{elkaroui1987,lacker2015} that $(t,x,\bar{\mu})\mapsto v(t,x,\mu)$ is 
	continuous\footnote{The main idea is to use Berge Theorem~\cite[Theorem~17.31]{aliprantis2006}. To this end, it suffices to show that the set-valued map $(t,x,\bar{\mu})\mapsto \Rc(t,x,\mu)$ has closed graph, nonempty compact values and that any control rule $\P\in\Rc(t,x, \mu)$ can be approximated by control rules $\P_n \in \Rc(t_n, x_n, \mu_n)$ whenever $(t_n, x_n, \bar{\mu}_n) \to (t,x, \bar{\mu})$.}
	on $[0,T]\x\R^d\x\Pc(\Dc).$
	We deduce that
	\begin{equation*}
		\E^{\Pb_n}\big[v\big(S_{k}, X^k_{S_k},\mu_n \big) \big] \xrightarrow[n\to\infty]{} \E^{\Pb}\big[v\big(S_{k}, X^k_{S_k},\mu \big) \big].
	\end{equation*}	
	Next a straightforward extension of Corollary~A.5 in~\cite{lacker2015} as above together with the continuity of the projection $Z_T$~\cite[Theorem 12.5]{billingsley2013} yields that
	\begin{equation*}
	 J_k\left(\mu_n, \Pb_n\right) \xrightarrow[n\to\infty]{} J_k\left(\mu, \Pb\right),
	\end{equation*}
	using the identity $\1_{T_k=T}=\langle Z_T,\1_{\{k\}}\rangle$ for all $k\in \bigcup_{t\in[0,T]} K_t$ to handle the indicator function.
	\qed

\subsubsection{Proof of Lemma~\ref{lem:psi}}
\label{sec:proof-compact}

\begin{Lemma}\label{lem:compact}
 The set $\Tc(\mu)$ is compact.
\end{Lemma}

\begin{proof}
	First we observe that the set $\Tc(\mu)$ is closed in view of the proof of Lemma~\ref{lem:closed-graph}. Next we  aim at showing that it is relatively compact in $\Pc(\bar{\Om})$. As $[0,T]\x\bar{A}$ is compact, so are $\bar{\Vc}$ and $\Pc(\bar{\Vc}).$ Thus it remains to check that $\{\Pb\circ Z^{-1}, \Pb\in\Tc(\mu)\}$ is relatively compact in $\Pc(\Dc)$. According to Theorem~2.1 in Roelly~\cite{roelly1986}, it suffices to show that, for all $\bar{\varphi}\in\Cc_b^2(\K\x\R^d)$, $\{\Pb\circ (Y^{\bar{\varphi}})^{-1}, \Pb\in\Tc(\mu)\}$ is tight where 
	\begin{equation*}
	Y^{\bar{\varphi}}_t := \sum_{k \in K_t} \varphi^k(X^k_t),\quad t\in[0,T].
	\end{equation*}

	Let $\bar{\varphi}\in\Cc_b^2(\K\x\R^d)$,
	it follows by Theorem~II.2.42 and Corollary~II.2.38 of Jacod and Shiryaev~\cite{jacod03} that $Y^{\bar{\varphi}}$ is a semimartingale,
	which admits the decomposition (w.r.t. the canonical filtration $\bar \F$)
  \begin{equation*}
  Y_t^{\bar{\varphi}} = Y_0^{\bar{\varphi}} + A^{\bar{\varphi}}_t + M^{\bar{\varphi}}_t,
  \end{equation*}
  where $A^{\bar{\varphi}}$ is a process with finite variation given by
  \begin{multline*}
  A^{\bar{\varphi}}_t := \int_0^t \sum_{k\in K_s} {\int_{A} {{ \Lc^{\mu,a}_s \varphi^k(X^k_s)} \,\Lambda^k(ds,da)}} \\
  +  \int_0^t \sum_{k\in K_s} {\int_A {\gamma(s,X^k_s, \mu_s, a) \sum_{\ell\in\N} {p_{\ell}(s,X^k_s,\mu_s) \Big(\sum_{i=1}^{\ell}\varphi^{ki}(X^k_s) - \varphi^k(X^k_s)\Big)}}\, \Lambda^k(da, ds)},
  \end{multline*}
  and $M^{\bar{\varphi}}_t$ is a local martingale with quadratic variation 
  \begin{multline*}
		\langle M^{\bar{\varphi}}\rangle_t = \int_0^t \sum_{k\in K_s} {\int_{A} {\big|D\varphi^k(X^k_s) \sigma(s,X_s^k,\mu_s,a)\big|^2} \,\Lambda^k(ds,da) } \\
		+ \int_0^t \sum_{k \in K_s} \int_A {\gamma(s,X^k_s, \mu_s, a) \sum_{\ell\in\N} {p_{\ell}(s,X^k_s,\mu_s) \Big(\sum_{i=1}^{\ell}\varphi^{ki}(X^k_s) - \varphi^k(X^k_s)\Big)^2}  \, \Lambda^k(da, ds)}. 
  \end{multline*}  
 Then we consider the localized process $Y^{\bar{\varphi},n}:=Y^{\bar{\varphi}}_{\cd\we\t_n}$ with $\t_n:= \inf\{t\ge 0;\, N_t\ge n\}$
 	and we observe that 
\begin{equation*}
 V(A^{\bar{\varphi}})_{t\wedge\tau_n} + \langle M^{\bar{\varphi}}\rangle_{t\wedge\tau_n} \leq C t,
\end{equation*}
where $V(A^{\bar{\varphi}})$ stands for the total variation of $A^{\bar{\varphi}}$. 
Thus it follows from Theorem~2.3 of Jacod \textit{et al.}~\cite{jacod83} that the family $\{ \Pb \circ (Y^{\bar{\varphi},n})^{-1};\, \Pb \in \Tc(\mu) \}$ is tight.
To conclude, it remains to notice that, in view of Lemma \ref{lem:Nt},
	\begin{equation*}
		\sup_{\Pb\in\Tc(\mu)} {\Pb \big(\tau_n \leq T \big)} \le \frac{C}{n},
	\end{equation*}
	and so the family $\{ \Pb \circ (Y^{\bar{\varphi}})^{-1};\, \Pb \in \Tc(\mu) \}$ is also tight.
	\qed
\end{proof}

	\noindent {\bf Proof of Lemma~\ref{lem:psi}}
	Assume first that the coefficients $(p_\ell)_{\ell\in\N}$ do not depend on $\mu.$
	Then the proof is a direct consequence of Lemma~\ref{lem:compact} above. 
	Indeed, let $C:=\|b\|\vee\|\sigma\| \vee \|\gamma\|$ and consider the case $A=[-C,C]^d\x[-C,C]^{d\x d} \x [0, C]$ and 
	\begin{equation*}
		b(t,x,\mu,a) = \beta,
		~~\sigma(t,x,\mu,a) = \upsilon,
		~~ \gamma(t,x,\mu,a) = \eta,
		\quad\mbox{for all}~ a=(\beta,\upsilon, \eta)\in A.
	\end{equation*}
	Denote by $\bar{\Tc}$ be the set of probability measures in Definition \ref{def:tree} with the above coefficient $(b, \sigma, \gamma)$,
	which does not depend on $\mu.$
	Notice that $\bar{\Tc}$ is convex by definition and that it is compact in view of Lemma~\ref{lem:compact}.
	To conclude, it remains to observe that
	\begin{equation*}
		 \bigcup_{\mu\in\Pc(\Dc)} {\Tc(\mu)} \subset \bar{\Tc}.
	\end{equation*}
	In the general case, when $(p_\ell)_{\ell\in\N}$ depend on $\mu,$ the proof follows by the same arguments using a straightforward extension of Lemma~\ref{lem:compact} where the coefficient $(p_\ell)_{\ell\in\N}$ are being controlled.  \qed

\subsubsection{Proof of Proposition~\ref{prop:non-empty-system}}
\label{sec:proof-empty}

	\noindent {\bf Proof of Proposition~\ref{prop:non-empty-system}.}
	$\mathrm{(i)}$ First, we observe that 
	\begin{equation*}
		 \Tc^*(\mu)=\bigcap_{n\in\N}{\Tc^*_n(\mu)},
	\end{equation*}
	where $\Tc^*_n(\mu)$ is the collection of  $\Pb\in\Tc(\mu)$ satisfying the optimality constraints~\eqref{eq:def_T_star} up to the $n$th generation, \ie,
	\begin{equation*}
		\Tc^*_n(\mu)
		:=
		\Big\{
		\Pb\in\Tc(\mu) ~:
		J_k\left(\mu,\Pb\right)
		=
		\E^{\Pb}\left[v\big(S_{k}, X^k_{S_k},\mu \big) \right],
		~k\in\bigcup_{i=1}^n{\N^i}
		\Big\}.
	\end{equation*}
	Notice that the set $\Tc^*_n(\mu)$ is closed in view of the proof of Lemma~\ref{lem:closed-graph}. Together with Lemma~\ref{lem:psi}, it follows that $\Tc^*_n(\mu)$ is compact.  
	It remains to show that $\Tc^*_n(\mu)$ is non-empty in order to conclude by Cantor's intersection theorem.
	
	\vspace{0.5em}
	
	\noindent $\mathrm{(ii)}$
	To show that $\Tc^*_n(\mu)$ is nonempty,
	we use an induction argument to construct an optimal tree up to the $(n+1)$th generation by concatenation of an optimal relaxed control for the first particle and an optimal tree for the subsequent $n$ generations.  
	To this end, given $(t,e)\in[0,T]\x E,$ we introduce the set $\Tc^*_n(t,e,\mu)$ as the collection of $\Pb\in\Pc(\bar{\Om})$ such that
	\begin{itemize}
		\item[(a)] $\Pb(Z_s = e,\, s \le t )=1$;
  
		\item[(b)] for all $\Phi\in\Cc^2_b(\R)$, $\bar{\varphi}\in\Cc_b^2(\K\x\R^d)$, the process
		\begin{equation*}
			M_s^{t,\mu,\Phi_{\bar{\varphi}}}(\bar{\om}) := \Phi_{\bar{\varphi}}\left(z(s)\right) - \int_t^s \!\!\int_{\bar{A}} {\cH^{\mu,\bar{a}}_r {\Phi_{\bar{\varphi}}}\left(z(r)\right) \,\bar{\lambda}(dr, d\bar{a})},\quad s\in [t,T],
		\end{equation*} 
		 is a $\Pb$--martingale;
  
		\item[(c)] for all $k=k' k''$ such that $k'\in K_0$ and $k''\in\bigcup_{i=1}^{n-1}\N^i$,
		\begin{equation*}
			J_k\left(\mu,\Pb\right)
			=
			\E^{\Pb}\left[v\big(S_k, X^k_{S_k},\mu \big) \right].
		\end{equation*}
	\end{itemize}
	Let us show that $\Tc^*_n(t,e,\mu)$  is non-empty for all $(t,e)\in[0,T]\x E$ by induction.
	
	\vspace{0.5em}

	\noindent $\mathrm{(ii.1)}$
	For the base case $n=0,$ we observe that there is no optimality constraint~(c). Thus, existence of an element in $\Tc(t,e,\mu):=\Tc^*_0(t,e,\mu)$ follows from a straightforward extension of Proposition~\ref{prop:dsm} with initial condition $Z_s=e$ for all $s\leq t$. 

	\vspace{0.5em}

	\noindent $\mathrm{(ii.2)}$ For the induction step, 
    let us construct first an element in $\Tc^*_{n+1}(t,e_k, \mu)$ with $e_k := \delta_{(k, x^k)}.$
	To this end, we take an optimal relaxed control $\P \in \Rc^*(t,x_k,\mu)$ and we define an integer-valued random variable $I$ as follows: 
    \begin{equation*}
     I :=  \int_{(t,\tau]\x \R_+} {\sum_{\ell\in\N} {\ell \,\mathbf{1}_{z\in I_\ell(s,X_s,\Lambda_s)}} \,Q(ds,dz)} = \sum_{\ell\in\N} {\ell \,\mathbf{1}_{U\in I_\ell(\tau,X_\tau,\Lambda_\tau)}},
    \end{equation*}
    where $(\tau,U)$ belongs to the support of $Q$ and for all $(s,x,\kappa)\in[0,T]\x\R^d\x\Pc(A),$
	\begin{equation*}
	 I^\mu_\ell(s,x,\kappa) := \Big[\int_A {\gamma(s,x,\mu_s, a) \,\kappa(da)} \sum_{i=0}^{\ell-1}{p_i(s,x,\mu_s)}, \int_A {\gamma(s,x,\mu_s, a) \,\kappa(da)} \sum_{i=0}^{\ell}{p_i}(s,x,\mu_s)\Big).
	\end{equation*}
    Then we introduce the process $(Z^k_s)_{s\in [0,T]}$ by
	\begin{equation*}
		 Z^k_s := 
		 \begin{cases}
			\delta_{(k, X_s)}, & \text{if } s\in[0,\tau), \\
			\sum_{i=1}^I \delta_{(ki, X_{\tau})}, & \text{if } s\in[\tau,T],
		 \end{cases}
	\end{equation*}
	and the random measure $\bar{\Lambda}_k$ as the pushforward of $\Lambda$ by the map 
\begin{equation*}
 (s,a)\in [0,T]\x A \mapsto (s,\bar{a})\in [0,T]\x\bar{A}, 
 \quad \text{where } \bar{a}(k') = \begin{cases}
  a & \text{if } k'=k, \\
  a_0 & \text{otherwise,}
 \end{cases}
\end{equation*} 
for some fixed $a_0\in A.$
Let us show that the process 
\begin{equation*}
	\Phi_{\bar{\varphi}}(Z^k_{s\wedge\tau}) -  \int_t^{s\wedge\tau} \!\!\int_{\bar{A}} {\cH^{\mu,a}_r {\Phi_{\bar{\varphi}}}(Z^k_r) \,\bar{\Lambda}_k(dr, d\bar{a})},
\end{equation*}
is a $\P$-martingale on $[t,T]$.
As $\P \in \Rc(t,x_k,\mu)$, the only difficulty is to check that
\begin{multline*}
\Phi_{\bar{\varphi}}(Z^k_{s\wedge\tau}) - \Phi_{\bar{\varphi}}(\delta_{(k,X_{s\wedge\tau})}) \\
- \int_t^{s\wedge\tau} \!\! \int_A {\gamma(r,X_r, \mu_r, a) \Big(\sum_{\ell\in\N} { \Phi_{\bar{\varphi}}\Big(\sum_{i=1}^{\ell} \delta_{(ki,X_r)} \Big) p_\ell(r,X_r,\mu_r)  - \Phi_{\bar{\varphi}}\big(\delta_{(k,X_r)}\big) } \Big) \Lambda(da, dr)}
\end{multline*}
is a $\P$--martingale. 
This property follows immediately from the identity
	\begin{equation*}
		\Phi_{\bar{\varphi}}(Z^k_{s\wedge\tau}) - \Phi_{\bar{\varphi}}(\delta_{(k,X_{s\wedge\tau})})
		= \int_{(t,s\wedge\tau]\x\R_+} {\sum_{\ell\in\N} \Big(\Phi_{\bar{\varphi}}\Big(\sum_{i=1}^\ell \delta_{(ki, X_r)}\Big) - \Phi_{\bar{\varphi}}\big(\delta_{(k,X_r)}\big)\Big) \1_{z\in I_\ell(r,X_r,\Lambda_r)} \,Q(dr,dz)}.
	\end{equation*}
In other worlds, if we denote $\Pb_k := \P \circ (Z^k,\bar{\Lambda}_k)^{-1}$, the process $M^{t,\mu,\Phi_{\bar{\varphi}}}$ is a $\Pb_k$--martingale on $[t,\tau_1]$ where $\eta$ is the first branching time given as
\begin{equation}\label{eq:branching-time}
 \eta := \inf \left\{s> 0;\, K_s\neq K_0\right\}.
\end{equation}
	Furthermore, by induction hypothesis and Lemma~\ref{lem:tc_mu_non_empty} below, we can use a measurable selection argument~\cite[Theorem~12.1.10]{stroock-varadhan-79} to ensure existence of a Borel map
   \begin{equation*}
    (s,\tilde{e})\in[0,T]\x E\mapsto \Pb^n_{s,\tilde{e}}\in\Tc^*_n(s,\tilde{e},\mu).
   \end{equation*}
	This allows us to define by 
	concatenation\footnote{The concatenation $\Pb^*_{k} := \Pb_k \otimes_{\eta} \Pb^{n}_{\cdot}$ is defined as the probability measure satisfying
	for all $\phi: \Dc \to \R$ and $\varphi: [0,T]\x \bar{A} \to \R$ bounded,
	\begin{multline*}
		\E^{\Pb^*_{k}}\Big[ \phi(Z) \int_0^T \!\! \int_{\bar{A}} \varphi(t,\bar{a}) \,\bar{\Lambda}(dt, d\bar{a}) \Big]
		:=
		\int_{\Omb} \E^{\Pb^{n}_{\eta(\omb),Z_{\eta}(\omb)}} \Big[ \phi(Z) \int_{\eta(\omb)}^T\int_{\bar{A}} \varphi(t,\bar{a}) \,\bar{\Lambda}(dt, d\bar{a})  \Big] \,\Pb_k(d \omb) \\
		+ 
		\int_{\Omb} \E^{\Pb^{n}_{\eta(\omb), Z_{\eta}(\omb)}} \big[ \phi(Z) \big] \int_0^{\eta(\omb)}\!\! \int_{\bar{A}} \varphi(t,\bar{a}) \,\bar{\Lambda}(\omb)(dt, d\bar{a})  \,\Pb_k(d \omb).
	\end{multline*}
	See, \eg, El Karoui and Tan~\cite[Section 4]{karoui2013capacities} for more details.
}
	$\Pb^*_{k} := \Pb_k \otimes_{\eta} \Pb^{n}_{\cdot}$. In view of~\cite[Theorem 1.2.10]{stroock-varadhan-79}, the process $M^{t, \mu, \Phi_{\bar \varphi}}$ is a $\Pb^*_k$--local martingale.
	It is actually a martingale by Lemma \ref{lem:Nt} since
	\begin{equation*}
	 \sup_{s \in [t,T]} \big|M_{s}^{t, \mu, \Phi_{\bar \varphi}}\big| \leq C\Big(1 + \sup_{s\in[t,T]} {\{N_s\}}\Big).
	\end{equation*}
	We conclude that $\Pb^*_{k} \in \Tc^*_n(t,e_k, \mu)$ by Lemma~\ref{lem:optimal_controlled_branching} below.

	\vspace{0.5em}
	
	\noindent $\mathrm{(ii.3)}$ Next we show existence of an element in $\Tc^*_{n+1}(t,e, \mu)$ with $e := \sum_{k\in K} {e_k}.$
	Let us introduce the probability measure $\Pb^*$ as the pushforward  of $\bigotimes_{k\in K} \Pb^*_{k}$ by the map 
	\begin{equation*}
		 \big(z^k,\bar{\lambda}_k)_{k\in K}\in\Omb^K \mapsto \Big(\sum_{k\in K} {z^k}, \bar{\lambda}\Big)\in\Omb, 
	\end{equation*}
	where $\bar{\lambda}$ is the pushforward of $(\bar{\lambda}^k)_{k\in K}$ by the map 
\begin{equation*}
 \big(s, (\bar{a}_k)_{k\in K} \big)\in [0,T]\x \bar{A}^K \mapsto (s,\bar{a})\in [0,T]\x\bar{A}, 
 \quad \text{where } \bar{a}(k') = \begin{cases}
  \bar{a}_k(k') & \text{if } k'\succeq k, \\
  a_0 & \text{otherwise.}
 \end{cases}
\end{equation*} 	
	 Then we consider the process $Y^k_s := \sum_{k' \in K^k_s} \varphi^{k'}(X^{k'}_s)$ where $K^k_s:=\{k' \in K_s;\, k' \succeq k \}$.
	In view of Remark~\ref{rem:MartingalePb}, it is a $\Pb^*$--semimartingale with characteristics 
	\begin{equation*}
		\big(\bar{B}^k_s, \bar{C}^k_s, \bar{\nu}^k(ds,dy)\big)
		:=
		\sum_{k' \in K^k_s} \big( B^{k'}_s, C^{k'}_s, \nu^{k'}(ds,dy) \big).
	\end{equation*}
	Since $(Y^k)_{k \in K}$ are $\Pb^*$--independent,
	the process $Y_s := \sum_{k \in K} Y^k_s$ is also a $\Pb^*$--semimartingale with characteristics 
	\begin{equation*}
		\big(\bar{B}_s, \bar{C}_s, \bar{\nu}(ds,dy)\big) := \sum_{k \in K} \big(\bar{B}^k_s, \bar{C}^k_s, \bar{\nu}^k(ds,dy)\big)
		= \sum_{k' \in K_s} \big(B^{k'}_s, C^{k'}_s, \nu^{k'}(ds,dy)\big).
	\end{equation*}
	Therefore, $M^{t, \mu, \Phi_{\bar \varphi}}$ is a $\Pb^*$--martingale by Remark~\ref{rem:MartingalePb}. 
	We conclude that $\Pb^*\in\Tc^*_{n+1}(t,e,\mu)$ since for all $k'\succeq k\in K,$
	\begin{equation*}
	 J_{k'}\left(\mu,\Pb^*\right) = J_{k'}\left(\mu,\Pb^*_k\right) = \E\big[v\big(S_{k'},X^{k'}_{S_{k'}},\mu\big)\big].
	\end{equation*}
	
	\vspace{0.5em}
	
	\noindent $\mathrm{(ii.4)}$
	Finally, to construct an element in $\Tc^*_n(\mu)$, 
	it remains to choose a measurable family $(\Pb^n_x)_{x \in \R^d}$ such that $\Pb^n_x \in \Tc^*_n(0, \delta_{(1,x)}, \mu)$. Then one can easily check that
	\begin{equation*}
		\int_{\R^d} {\Pb^n_x \,m_0(dx)} \in \Tc^*_n(\mu).
	\end{equation*}	
	\qed

	\begin{Lemma}\label{lem:tc_mu_non_empty}
	The set--valued map $(t,e)\mapsto\Tc^*_n(t,e,\mu)$  has a closed graph.
	\end{Lemma}
	
	\begin{proof}
	The proof follows by similar arguments as Lemma~\ref{lem:closed-graph}.
	The only difficulty concerns the initial condition~(a) as the projection $Z_t$ is not continuous for the Skorokhod  topology.
	Consider a sequence $(t_m,e_m,\Pb_m)_{m\in\N}$ converging to $(t,e,\Pb)$ such that $\Pb_m\in\Tc^*_n(t_m, e_m, \mu).$ We aim to prove that $\Pb(Z_{t} = e) = 1.$ Let $\Phi\in\Cc_b^2(\R)$ and $\bar{\varphi}\in\Cc^2_b(\K\x\R^d)$ be chosen arbitrarily. 
	First we observe that 
	\begin{equation*}
	\bigg| \frac{1}{h} \E^{\Pb_m} \Big[ \int_{t}^{t+h} {\Phi_{\bar{\varphi}} (Z_s) \,ds} \Big] - \Phi_{\bar{\varphi}} (e) \bigg| \leq \bigg| \frac{1}{h} \E^{\Pb_m} \Big[ \int_{t}^{t+h} {\big(\Phi_{\bar{\varphi}} (Z_s) - \Phi_{\bar{\varphi}} (e_m) \big)\,ds} \Big]  \bigg| + \big|\Phi_{\bar{\varphi}} (e_m) - \Phi_{\bar{\varphi}} (e)\big|.
	\end{equation*}
	Then, denoting by $\eta$ the first branching time as in~\eqref{eq:branching-time},  it follows by the martingale property and Lemma~\ref{lem:Nt} that
	\begin{equation*}
	 \bigg| \frac{1}{h} \E^{\Pb_m} \Big[ \int_{t}^{t+h} {\big(\Phi_{\bar{\varphi}} (Z_s) - \Phi_{\bar{\varphi}} (e_m)\big) \,ds}\,  \mathbf{1}_{\eta > t}\Big] \bigg| \leq C h.
	\end{equation*}
	Notice that $\Pb_m(\eta>t)=1$ whenever $t_m\geq t$ by definition of $\Tc^*_n(t_m, e_m, \mu).$ In the case $t_m<t,$ we can use further Cauchy--Schwartz Inequality and Lemma~\ref{lem:Nt} to obtain that
	\begin{equation*}
	 \bigg| \frac{1}{h} \E^{\Pb_m} \Big[ \int_{t}^{t+h} {\big(\Phi_{\bar{\varphi}} (Z_s) - \Phi_{\bar{\varphi}} (e_m)\big) \,ds}\,  \mathbf{1}_{\eta \leq t}\Big] \bigg| \leq C \sqrt{\Pb_m\left(\eta\leq t\right)}.
	\end{equation*}	
	In addition, it holds
	\begin{equation*}
	\Pb_m\left(\eta\leq t\right) \leq \sum_{k\in K_m} \Pb_m\left(T_k\leq t\right).
	\end{equation*}
	Writing the martingale problem for $\Phi=\mathrm{Id}$ and $\bar{\varphi}=\mathbf{1}_{k}$ with $k\in K_m$ and taking expectation, we obtain that 
	\begin{equation*}
	\Pb_m\left(T_k> t\right) = 1 - \int_{t_m}^t {\E^{\Pb_m}\Big[\int_A {\gamma\left(s,X_s,\mu_s,a\right)}\,\Lambda^k_s(da)\,\mathbf{1}_{T_k>s}\Big] \, ds}
	\end{equation*}
	We deduce by the differential form of Gr\"onwall's lemma that  
	\begin{equation*}
	\Pb_m\left(T_k> t\right) \geq e^{-\|\gamma\| (t_m-t)}.
	\end{equation*}	
	We conclude that 
	\begin{equation*}
	 \bigg| \frac{1}{h} \E^{\Pb_m} \Big[ \int_{t}^{t+h} {\Phi_{\bar{\varphi}} (Z_s) \,ds} \Big] - \Phi_{\bar{\varphi}} (e) \bigg| \leq C \left(h + \sqrt{(t_m-t)_+}\right) + \big|\Phi_{\bar{\varphi}} (e_m) - \Phi_{\bar{\varphi}} (e)\big|,
	\end{equation*}	
	which yields the desired by letting $m\to\infty$ and $h\to 0.$
	\qed
\end{proof}

\begin{Lemma}\label{lem:optimal_controlled_branching}
 The probability measure $\Pb\in\Tc^*(\mu)$ if and only if $\Pb\in\Tc(\mu)$ and  for all $k\in\K,$
 \begin{equation*}
 		\E^{\Pb}\Big[ \int_{S_k}^{T_k} \!\! \int_A {f\big(s, X^k_s, \mu_s, a\big) \,\Lambda^k(ds,da)} + g\big(X^k_T,\mu_T\big) \mathbf{1}_{T_k= T} ~\Big|\, \bar{\Fc}_{S_k} \Big] = v\big(S_{k}, X^k_{S_k},\mu\big).
 \end{equation*}
\end{Lemma}

\begin{proof}
	Let $k \in \K$ and $\Pb \in \Tc(\mu)$,
	denote by $(\Pb_{\omb})_{\omb \in \Omb}$ a family of conditional probability measures of $\Pb$ w.r.t. $\bar{\Fc}_{S_k}$.
	Then, for $\Pb$--a.e. $\omb \in \Omb$, the process $(M^{\mu,\Phi_{\bar{\varphi}}}_s)_{s \in[S_k(\omb), T]}$ is a $\Pb_{\omb}$--martingale for all $\Phi \in \Cc^2_b(\R)$ and $\bar{\varphi}\in\cC^2_b(\K\x\R^d)$.
	By setting  $\Phi(x) = x$, $\varphi^k = 1 $ and $\varphi^{k'} = 0$ otherwise,
	it follows that the process $t \mapsto \1_{t<T_k}$ is non-increasing with $\Pb_{\omb}$--compensator 
	\begin{equation*}
	 -\int_{S_k(\omb)}^t \int_A {\gamma(s, X_s,\mu_s, a) \1_{s < T_k}}\, \Lambda^k(ds, da).
	\end{equation*}
	By the representation theorem of semimartingales~\cite[Theorem II.7.4]{ikeda89}, there exists a Poisson random measure $Q^k(dt,dz)$ on $[S_k(\omb),T] \x \R_+$ with compensator $dt \, dz$  such that
	\begin{equation*}
		T_k = \inf\Big\{ s > S_k(\omb);\, Q^k\Big(\{s\} \x \Big[0, \int_A \gamma(s, X_s, \mu_s, a) \Lambda^k_s(da) \Big]\Big) = 1 \Big\},\quad \Pb_{\omb} \mbox{--a.s.}
	\end{equation*}
	Then considering again the martingale problem with $\Phi(x) = x$, $\varphi^k = \varphi \in C^2_b(\R^d)$ 
	and $\varphi^{k'} = 0$ otherwise,
	we observe that the law of $(X^k, \Lambda^k, Q^k)$ under $\Pb_{\omb}$ induces an element of $\Rc(S_k(\omb),X_{S_k}(\omb),\mu)$.
	In particular, it holds
	\begin{equation*}
		\E^{\Pb}\Big[ \int_{S_k}^{T_k} \!\! \int_A {f\big(s, X^k_s, \mu_s, a\big) \,\Lambda^k(ds,da)} + g\big(X^k_T,\mu_T\big) \mathbf{1}_{T_k= T} ~\Big|\, \bar{\Fc}_{S_k} \Big] \le v\big(S_{k}, X^k_{S_k},\mu\big).
	\end{equation*}
	Thus, if $\Pb\in\Tc^*(\mu)$, it follows from the optimality constraint~\eqref{eq:def_T_star} that
	\begin{equation*}
		\E^{\Pb}\Big[ \int_{S_k}^{T_k} \!\! \int_A {f\big(s, X^k_s, \mu_s, a\big) \,\Lambda^k(ds,da)} + g\big(X^k_T,\mu_T\big) \mathbf{1}_{T_k= T} ~\Big|\, \bar{\Fc}_{S_k} \Big] = v\big(S_{k}, X^k_{S_k},\mu\big).
	\end{equation*}
\qed
\end{proof}

\begin{appendix}
\section{Proof of Lemma~\ref{lem:technicpde}}
\label{app:proof-lemma}

Throughout this section, $C$ denotes a positive constant depending solely on $T$, $d$, $\|\gamma\|$, $\|\sum_{\ell\in\N}{\ell p_{\ell}}\|$ which may change from line to line.

\no\rmi In view of Proposition~\ref{prop:def}, it holds
    \begin{equation}\label{eq:moment_app}
     m(t)\big(\R^d\big)= \E\left[N_t\right] \leq C.
    \end{equation}
Hence it suffices to show that
\begin{equation}\label{eq:moment_fp}
 \int_{\R^d} {|x|^2\, m(t)(dx)} \leq C\left(1 + \|b\|^2 + \int_{\R^d} {|x|^2\, m_0(dx)}\right).
\end{equation}
Denote
\begin{equation*}
 M_t := \sum_{k\in K_t} \!\! {\big|X^k_t\big|^2} -  \left|X_0\right|^2 
	- \int_0^{t} {\!\!\sum_{k\in K_s} \Big( 2 d + 2 b(s,X^k_s)\cdot X_s^k + \gamma(X^k_s)\sum_{\ell\in\N}{(\ell-1)p_{\ell}(X^k_s)} \big|X^k_s\big|^2 \Big) ds}.
\end{equation*}
It follows from Proposition~\ref{prop:dsm} that $M$ is a local martingale with localizing sequence of stopping times
\begin{equation*}
 \tau_n:=\inf{\Big\{t\geq 0;\  \sup_{k\in K_t}|X_t^k|\geq n\Big\}}.
\end{equation*}
Furthermore, a straightforward calculation yields that 
\begin{equation*}
 	\sum_{k\in K_{t\wedge\tau_n}} \!\! {\big|X^k_{t\wedge\tau_n}\big|^2}
	\leq  \left|X_0\right|^2 + \int_0^{t\wedge\tau_n} {\!\!\sum_{k\in K_s} \Big(2d + \|b\|^2 + \Big(1+\Big\|\gamma\sum_{\ell\in\N}{\ell p_{\ell+1}}\Big\|\Big) \big|X^k_s\big|^2 \Big)\,ds} + M_{t\wedge\tau_n}.
\end{equation*}
Taking expectation and using~\eqref{eq:moment_app}, we obtain
\begin{equation*}
 \E\bigg[\sum_{k\in K_{t\wedge\tau_n}} {\big|X^k_{t\wedge\tau_n}\big|^2}\bigg] \leq C\bigg(1 + \left\|b\right\|^2 + \E\big[|X_0|^2\big] + \int_0^{t} {\E\bigg[ \sum_{k\in K_{s\wedge\tau_n}} \big|X^k_{s\wedge\tau_n}\big|^2\bigg] \,ds}\bigg).
\end{equation*}
By Gr\"onwall's lemma, we deduce that
\begin{equation*}
 \E\bigg[\sum_{k\in K_{t\wedge\tau_n}} {\big|X^k_{t\wedge\tau_n}\big|^2}\bigg] \leq  C\Big(1 + \left\|b\right\|^2 + \E\big[|X_0|^2\big]\Big).
\end{equation*}
The conclusion then follows from Fatou's lemma.

\no\rmii Let us show next that for all $0\leq t\leq s\leq T$, 
\begin{equation}
 W_1\big(m(t),m(s)\big)\leq C\left(1 + \|b\|^2 + \int_{\R^d} {|x|^2\,  m_0(dx)}\right)\sqrt{s-t}.
\end{equation}
Using Kantorovitch's duality (see Lemma~\ref{lem:duality}), we have
\begin{eqnarray*}
 	W_1(m(t),m(s)) &=& \sup{\left\{\int_{\R^d} {\varphi(x)\, m(t)(dx)} - \int_{\R^d} {\varphi(x)\, m(s)(dx)}\right\}}
 	~+~\left|m(t)(\R^d)-m(s)(\R^d)\right|,
\end{eqnarray*}
where the supremum is taken over the set of all $1$-Lipschitz continuous maps $\varphi:\R^d\to\R$ such that $\varphi(0)=0$. For the second term on the r.h.s., we observe that
\begin{equation*}
 \E[N_s] = \E[N_t] + \int_t^s {\E\left[\sum_{k\in K_r} \gamma(X^k_r)\sum_{\ell\in\N}(\ell-1)p_\ell(X^k_r)\right] \,dr},
\end{equation*}
We deduce by~\eqref{eq:moment_app} that
\begin{equation*}
 \big|m(t)(\R^d)-m(s)(\R^d)\big| = \big|\E[N_t] - \E[N_s]\big|  \leq C (t-s),
\end{equation*}
As for the first term, let $\varphi:\R^d\to\R$ be an arbitrary
$1$-Lipschitz continuous map  satisfying $\varphi(0)=0$. It holds
\begin{equation}\label{eq:kantorovitch}
\left|\int_{\R^d} {\varphi(x)\, m(t)(dx)} - \int_{\R^d} {\varphi(x)\, m(s)(dx)}\right|
 = \bigg|\E\bigg[\sum_{k\in K_t} \varphi\big(X^k_t\big) - \sum_{k\in K_s} \varphi\big(X^k_s\big)\bigg]\bigg|. 
\end{equation}
We decompose the integrand on the r.h.s. as follows:
\begin{multline*}
 \sum_{k\in K_t} \varphi\big(X^k_t\big) - \sum_{k\in K_s} \varphi\big(X^k_s\big) = \sum_{k\in K_t\cap K_s} {\Big(\varphi\big(X^k_t\big) - \varphi\big(X^k_s\big)\Big)}
 + \sum_{k\in K_t\setminus K_s} {\varphi\big(X^k_t\big)} - \sum_{k\in K_s\setminus K_t} {\varphi\big(X^k_s\big)}.
\end{multline*}
First we observe that 
\begin{equation*}
 \bigg|\E\bigg[\sum_{k\in K_t\cap K_s} {\Big(\varphi\big(X^k_t\big) - \varphi\big(X^k_s\big)\Big)}\bigg]\bigg| 
 \leq \E\bigg[\sum_{k\in K_t} {\big| X^k_t - X^k_s \big|}\bigg], 
\end{equation*}
where, using the notations of Section~\ref{sec:branching-diffusion},
\begin{equation*}
X^k_s = X^{k}_t + \int^s_t b(r,X^k_r) \,dr + \sqrt{2}\big(B^k_s-B^k_t\big), \quad \P-\text{a.s.}
\end{equation*}
In particular, we have 
\begin{equation*}
 \E\Big[\big| X^k_t - X^k_s \big|\,\Big|\,\Fc_t\Big] \leq C(1+\|b\|)\sqrt{s-t}.
\end{equation*}
It then follows by~\eqref{eq:moment_app} that
\begin{equation*}
\bigg|\E\bigg[\sum_{k\in K_t\cap K_s} {\Big(\varphi\big(X^k_t\big) - \varphi\big(X^k_s\big)\Big)}\bigg]\bigg| \leq C\big(1+\|b\|\big) \sqrt{s-t}
\end{equation*}
Second, we observe that 
\begin{equation*}
 \bigg|\E\bigg[\sum_{k\in K_t\setminus K_s} {\varphi\big(X^k_t\big)}\bigg]\bigg| 
 \leq \E\bigg[\sum_{k\in \K} {\big|X^k_t\big|\mathbf{1}_{\{k\in K_t\setminus K_s\}}}\bigg].
\end{equation*}
In addition, we have 
\begin{equation}\label{eq:survival}
 \P\Big(k\in K_t\setminus K_s \,\Big|\, \Fc_t\Big) \leq \left(1 - e^{-\|\g\| (s-t)}\right)\mathbf{1}_{k\in K_t} \leq \|\g\| (s-t) \mathbf{1}_{k\in K_t},
\end{equation}
and it follows from Step~(i) above that 
\begin{equation}\label{eq:firstmoment}
 \E\bigg[\sum_{k\in K_t} {\big|X^k_t\big|}\bigg] \leq \frac{1}{2}\E\bigg[\sum_{k\in K_t} {\Big(1+\big|X^k_t\big|^2\Big)}\bigg] \leq C \Big(1+\|b\|^2+\E\big[|X_0|^2\big]\Big).
\end{equation}
Thus we deduce that 
\begin{equation*}
 \bigg|\E\bigg[\sum_{k\in K_t\setminus K_s} {\varphi\big(X^k_t\big)}\bigg]\bigg| \leq C \Big(1+\|b\|^2+\E\big[|X_0|^2\big]\Big) (s-t).
\end{equation*}
Third, we observe that
\begin{equation*}
 \bigg|\E\bigg[\sum_{k\in K_s\setminus K_t} {\varphi\big(X^k_s\big)}\bigg]\bigg| 
 \leq \E\bigg[\sum_{k\in K_t} {\mathbf{1}_{T_k\leq s} \sum_{\substack{ k'\in K_s \\ k'\succ k}}{\big|X^{k'}_s\big|}}\bigg].
\end{equation*}
In addition, similar to~\eqref{eq:firstmoment}, it holds
\begin{equation*}
 \E\bigg[\sum_{\substack{ k'\in K_s \\ k'\succ k}}{\big|X^{k'}_s\big|} \,\bigg|\,\Fc_{T_k}\bigg] \leq C \Big(1+ \|b\|^2 + \big|X^k_{T_k}\big|^2 \Big)
\end{equation*}
It follows that
\begin{align*}
 \bigg|\E\bigg[\sum_{k\in K_s\setminus K_t} {\varphi\big(X^k_s\big)}\bigg]\bigg| 
 & \leq C \E\bigg[\sum_{k\in K_t} {\mathbf{1}_{T_k\leq s} \Big(1+ \|b\|^2 + \big|X^k_{T_k}\big|^2 \Big)}\bigg] \\
 & \leq C \E\bigg[\sum_{k\in K_t\setminus K_s} {\Big(1+ \|b\|^2 + \big|X^k_t\big|^2  + \big|X^k_{T_k}-X^k_t\big|^2  \Big)}\bigg].
\end{align*}
To conclude, it remains to observe that, on the one hand, it follows by~\eqref{eq:survival}--\eqref{eq:firstmoment} that
\begin{equation*}
 \E\bigg[\sum_{k\in K_t\setminus K_s} {\Big(1+ \|b\|^2 + \big|X^k_t\big|^2\Big)}\bigg]\leq C\Big(1+\|b\|^2+\E\big[|X_0|^2\big]\Big)(s-t),
\end{equation*}
and, on the other hand, it holds
\begin{equation*}
 \E\bigg[\sum_{k\in K_t\setminus K_s} { \big|X^k_{T_k}-X^k_t\big|^2}\bigg] \leq \E\bigg[\sum_{k\in K_t} { \big|X^k_{T_k\wedge s}-X^k_t\big|^2}\bigg] \leq C\big(1+\|b\|^2\big)(s-t),
\end{equation*}
since
\begin{equation*}
 \E\Big[\big|X^k_{T_k\wedge s}-X^k_t\big|^2\,\Big|\,\Fc_t\Big] \leq C\big(1+\|b\|^2\big)(s-t).
\end{equation*}

\section{Wasserstein Distance for Finite Measures}
\label{sec:WasDis_App}

\def\Eb{\bar{E}}
\def\mub{\bar \mu}
\def\nub{\bar \nu}
\def\Xb{\bar{X}}

	Let $(X, d)$ be a nonempty Polish space and denote by $\Mc_1(X)$ the collection of all finite non-negative Borel measures on $X$ such that $\int_{X} d(x,x_0) \mu(dx) < \infty$ for some (and thus all) $x_0\in X$. We aim to define a Wasserstein type distance on $\Mc_1(X)$.
	
	To this end, we introduce a cemetery point $\partial$ to obtain an enlarged space $\Xb := X \cup \{\partial\}$.
	By defining $d(x, \partial) := d(x, x_0) + 1$, we extend the distance $d$ on $\Xb$ in such a way that $(\Xb, d)$ is still a Polish space.
	Next we introduce the classical Wasserstein distance on 
		\begin{equation*}
			\Mc_{1,m}(\Xb) := \{ \mu \in \Mc_1(\Xb);\ \mu(\Xb) = m \},
		\end{equation*}	
		as follows
	\begin{equation*}
		W_{1,m}(\mu, \nu)
		:= 
		\inf_{\pi \in \Pi(\mu, \nu)} \int_{\Xb \x \Xb} d(x,y) \,\pi(dx, dy),
		\quad \forall\, \mu, \nu \in \Mc_{1,m}(\Xb),
	\end{equation*}
	where $\Pi(\mu, \nu)$ denotes the collection of all non-negative measures on $\Xb \x \Xb$ with marginals $\mu$ and $\nu$.
	
	The Wasserstein type distance on $\Mc_1(X)$ is then defined as follows: 
	\begin{equation} \label{eq:Wp}
		W_1(\mu, \nu) := W_{1,m}(\mub_m, \nub_m),
		\qquad\forall\, \mu, \nu \in \Mc_1(X),
	\end{equation}
	where $m \ge \mu(X) \vee \nu(X)$ and 
	\begin{equation*} \label{eq:mu2mum}
		\mub_m(\cdot) := \mu ( \cdot \cap X) + \big(m - \mu(X) \big) \delta_{\partial}(\cdot),
		 \quad	
		\nub_m(\cdot) := \nu ( \cdot \cap X) + \big(m - \nu(X) \big) \delta_{\partial}(\cdot).
	\end{equation*}
	 As shown in the next lemma, this definition does not depend on the choice of $m$.

	\begin{Lemma}\label{lem:duality}
	The following dual representation holds: for all $\mu,\nu\in\Mc_1(X)$,
		\begin{equation*} \label{eq:W1Dual_App}
			W_1(\mu, \nu) 
			= 
			\sup_{\varphi \in \mathrm{Lip}^0_1(X)} \big\{ \mu(\varphi) - \nu(\varphi) \big\}
			+ \big| \mu(X) - \nu(X) \big|,			
		\end{equation*}
		where $\mathrm{Lip}^0_1(X)$ denote the collection of all functions $\varphi: X \to \R$ with Lipschitz constant smaller or equal to $1$ and such that $\varphi(x_0) = 0$. 
	\end{Lemma}
	
	\begin{proof}
	By Kantorovitch duality on $\Mc_{1,m}(\Xb)$, it holds 
		\begin{equation*}
			W_1(\mu, \nu) = W_{1,m}(\mub_m, \nub_m) = \sup_{\varphi \in \mathrm{Lip}^0_1(\Xb)} \big\{ \mub_m(\varphi) - \nub_m(\varphi) \big\}.
		\end{equation*}
		See, \eg, Villani~\cite[Remark 6.5]{villani09}. 
		Additionally, we have 
		\begin{equation*}
		\mub_m(\varphi) - \nub_m(\varphi) = \mu(\varphi|_X) - \nu(\varphi|_X) + \varphi(\delta)(\mu(X)-\nu(X))
		\end{equation*}
		To conclude, it remains to observe that $\varphi$ belongs to $\mathrm{Lip}^0_1(\Xb)$ if and only if $\varphi|_X$  belongs to $\mathrm{Lip}^0_1(X)$ and $|\varphi(\partial)|\leq 1$.
    \qed
	\end{proof}
	
	Next we show that $W_1$ defines a metric on the space of finite non-negative measures and collect important topological properties.

	\begin{Lemma} \label{lemm:Wp_App}		
	 $(\Mc_1(X),W_1)$ is a Polish space. Additionally, a sequence $(\mu_n)_{n\in\N}$ converges to $\mu$ in  $\Mc_1(X)$  if and only if, for all $\varphi: X \to \R$ continuous satisfying $|\varphi(x)| \le C(1+d(x, x_0))$,
		\begin{equation*}
			\int_X {\varphi(x) \,\mu_n(dx)} ~\longrightarrow~ \int_X {\varphi(x) \,\mu(dx)}.
		\end{equation*}
	\end{Lemma}
	
	\begin{proof}
	$\mathrm{(i)}$ The fact that $W_1$ is a metric on $\Mc_1(X)$ follows easily from the fact that $W_{1,m}$ is a metric on $\Mc_{1,m}(\Xb)$. Let us prove the triangular inequality for the sake of completeness. Let $\mu^1,\mu^2,\mu^3\in\Mc_1(X)$ and $m\geq\mu^i(X)$ for $i=1,2,3$. It holds
	\begin{equation*}
	 W_1(\mu^1,\mu^3) = W_{1,m}(\mub^1_m,\mub^3_m) \leq W_{1,m}(\mub^1_m,\mub^2_m) + W_{1,m}(\mub^2_m,\mub^3_m) = W_1(\mu^1,\mu^2) + W_1(\mu^2,\mu^3).
	\end{equation*} 
	
	\noindent $\mathrm{(ii)}$ Let us show next that $\Mc_1(X)$ is separable and complete under $W_1$. Recall that $\Mc_{1,m}(\Xb)$ is separable and complete for the classical Wasserstein distance. In particular, $\Mc_1(X)\subset \cup_{m \in\N} \Mc_{1, m}(\Xb)$ is separable as a subset of a separable space. To verify the completeness, we consider a Cauchy sequence $(\mu^k)_{k \in\N}$ in $\Mc_1(X)$. By Lemma~\ref{lem:duality}, $|\mu(X) - \nu(X)|\leq W_1(\mu, \nu)$ and thus $(\mu^k(X))_{k \in\N}$ is a Cauchy sequence in $\R_+$. In particular, it is bounded and so
	 $(\mu^ k)_{k \in \N}$ can be identified to a sequence $(\mub^k_m)_{k \in\N}$ in $\Mc_{1,m}(\Xb)$ for some $m$ large enough. Since $\Mc_{1,m}(\Xb)$ is complete, it converges to a measure $\mub \in \Mc_{1,m}(\Xb)$. 
	 It follows that $(\mu^k)_{k\in\N}$ converges to $\mu := \mub(\cdot \cap X)$ under $W_1$.
	 
	\noindent$\mathrm{(iii)}$ The characterization of the convergence follows easily from the analoguous result for the classical Wasserstein distance, see, \eg, Theorem 6.9 in~\cite{villani09}. Similar to the proof of completeness above, we can use the fact that the sequence $(\mu_n(X))_{n\in\N}$ is bounded to identify the sequence $(\mu_n)_{n\in\N}$ to a sequence of $\Mc_{1,m}(\Xb)$ for $m$ large enough. 
		\qed
	\end{proof}
	
	We conclude this section by providing a compactness criterion in the case $X=\R^d$.
	
	\begin{Lemma} \label{lemm:Wp_compact}		
	  Let $p > 1$. If $\Kc\subset \Mc_1(\R^d)$ satisfies
		\begin{equation*}
		 \sup_{\mu\in\Kc} { \left\{\int_{\R^d} \left(1 + |x|^p\right) \mu(dx) \right\}} < \infty,
		\end{equation*}
		then $\Kc$ is relatively compact.	
	\end{Lemma}

	\begin{proof}
 	 The set $\Kc$ can be identified to a subset of $\Mc_{1,m}(\Xb)$ for $m$ large enough. The conclusion then follows from the analogous result for the classical Wasserstein distance, see, \eg, Cardaliaguet \cite[Lemma 5.7]{cardaliaguet2010}.	 \qed
	\end{proof}

\end{appendix}

\bibliographystyle{abbrv}
\bibliography{bibliography}

\end{document}